\documentclass[11pt,reqno]{amsart}
\usepackage[T1]{fontenc}
\usepackage{lmodern}
\usepackage{amsmath,amssymb,mathtools}
\usepackage[margin=30mm]{geometry}
\usepackage{microtype}
\AtBeginDocument{\sbox0{$x$}}
\usepackage[colorlinks=true,linkcolor=blue,citecolor=blue,urlcolor=blue]{hyperref}
\hypersetup{pdftitle={Exponent of convergence for Hausdorff dimension of Dirichlet-improvable vectors},pdfauthor={Bohan Yang and Tianru Zhu}}
\numberwithin{equation}{section}
\newtheorem{theorem}{Theorem}[section]
\newtheorem{proposition}[theorem]{Proposition}
\newtheorem{lemma}[theorem]{Lemma}
\theoremstyle{definition}
\newtheorem{definition}[theorem]{Definition}
\newcommand{\R}{\mathbb R}
\newcommand{\Z}{\mathbb Z}
\newcommand{\N}{\mathbb N}
\newcommand{\eps}{\varepsilon}
\newcommand{\wt}{\mathbf w}
\newcommand{\DI}{\operatorname{DI}}
\newcommand{\Sing}{\operatorname{Sing}}
\newcommand{\covol}{\operatorname{covol}}
\newcommand{\vol}{\operatorname{vol}}
\newcommand{\dist}{\operatorname{dist}}
\newcommand{\Span}{\operatorname{span}}
\newcommand{\diag}{\operatorname{diag}}
\title[Convergence exponent for Dirichlet-improvable vectors]{Exponent of convergence for Hausdorff dimension of Dirichlet-improvable vectors}
\author{Bohan Yang}
\address{Shanghai Institute for Mathematics and Interdisciplinary Sciences,
Block A, International Innovation Plaza, 657 Songhu Road,
Yangpu District, Shanghai 200433, China}
\address{Research Institute of Intelligent Complex Systems, Fudan University,
220 Handan Road, Yangpu District, Shanghai 200433, China}
\email{bhyang@simis.cn}
\author{Tianru Zhu}
\address{Qiuzhen College, Tsinghua University, 30 Shuangqing Road,
Haidian District, Beijing 100084, China}
\email{ztr24@mails.tsinghua.edu.cn}
\date{}
\begin{document}
\begin{abstract}
Let $d\ge2$ and fix a positive weight vector
$\mathbf{w}=(w_1,\ldots,w_d)$ with $\sum_{i=1}^d w_i=1$.
We prove that the set of $\mathbf{w}$-singular vectors in
$\mathbb{R}^d$ has Hausdorff dimension $d-(1+\max_i w_i)^{-1}$.
As $\varepsilon\downarrow0$, the difference between the Hausdorff
dimension of the $\mathbf{w}$-weighted
$\varepsilon$-Dirichlet-improvable set and this value is of order
$\sqrt{\varepsilon}$ for $d=2$ and of order $\varepsilon$ for $d\ge3$.
\end{abstract}
\maketitle
\section{Introduction}

Fix $d\ge2$ and positive weights $\wt=(w_1,\ldots,w_d)$ satisfying
\begin{equation}\label{eq:weights}
 w_1\ge\cdots\ge w_d>0\quad\text{and}\quad \sum_{i=1}^d w_i=1.
\end{equation}
For $z=(z_1,\ldots,z_d)\in\R^d$, define the weighted quasi-norm
\[
 \|z\|_\wt=\max_{1\le i\le d}|z_i|^{1/w_i}.
\]
For $\eps>0$, let $\DI_\wt(\eps)$ be the set of
$\xi\in\R^d$ such that, for every sufficiently large $T$, there
exist $p\in\Z^d$ and $q\in\N$ satisfying
\begin{equation}\label{eq:DI}
 0<q<T\quad\text{and}\quad
 |q\xi_i-p_i|<(\eps/T)^{w_i}
\end{equation}
for each $i=1,\ldots,d$.
The set of $\wt$-singular vectors is
\[
 \Sing(\wt)=\bigcap_{\eps>0}\DI_\wt(\eps).
\]

\begin{theorem}\label{thm:singular}
Let $d\ge2$ and let $\wt$ satisfy \eqref{eq:weights}. Then
\[
 \dim_H\Sing(\wt)=d-\frac1{1+w_1}.
\]
\end{theorem}

Set $s_*=d-(1+w_1)^{-1}$. We also establish the following quantitative
dimension estimates.
\begin{theorem}\label{thm:main}
Let $d\ge2$ and let $\wt$ satisfy \eqref{eq:weights}. Then, as
$\eps\downarrow0$,
\[
 \dim_H\DI_\wt(\eps)-s_*
 \asymp_\wt
 \begin{cases}
  \sqrt\eps,&d=2,\\
  \eps,&d\ge3.
 \end{cases}
\]
The implied constants depend only on $d$ and $\wt$.
\end{theorem}

Since $\Sing(\wt)\subseteq\DI_\wt(\eps)$ for every $\eps>0$,
Theorem~\ref{thm:singular} follows from the upper bound in
Theorem~\ref{thm:main} together with the lower bound
$\dim_H\Sing(\wt)\ge s_*$ of Kim--Park \cite[Theorem~1.1]{KP}.

The study of singular vectors goes back to Khintchine \cite{Khi26}.
Dani \cite{Dani85} characterized singularity by divergence in the space
$\mathrm{SL}_{d+1}(\R)/\mathrm{SL}_{d+1}(\Z)$ of unimodular lattices: for
\begin{equation}\label{eq:weighted-flow}
 g_t=\diag(e^{w_1t},\ldots,e^{w_dt},e^{-t}),\quad
 h_\xi=\begin{pmatrix}I_d&\xi\\0&1\end{pmatrix},\quad\text{and}\quad
 \Lambda(\xi,t)=g_th_{-\xi}\Z^{d+1},
\end{equation}
a vector $\xi\in\R^d$ belongs to $\Sing(\wt)$ if and only if
$\Lambda(\xi,t)$ leaves every compact subset as $t\to\infty$.
Using weighted flows, Kleinbock--Weiss \cite[Theorem~1.4]{KW08} proved
that Dirichlet-improvable matrices (and in particular $\DI_\wt(\eps)$ for
$0<\eps<1$) have Lebesgue measure zero, whereas Kleinbock--Rao
\cite[Theorem~1.2]{KR22} showed that $\bigcup_{0<\eps<1}\DI_\wt(\eps)$
contains the $\wt$-badly approximable vectors and hence has full Hausdorff
dimension.

In the unweighted case, Cheung \cite[Theorem~1.1]{Cheung} proved
$\dim_H\Sing(\tfrac12,\tfrac12)=4/3$, and Cheung--Chevallier
\cite[Theorem~1.1]{CC} obtained
$\dim_H\Sing(\tfrac1d,\ldots,\tfrac1d)=d^2/(d+1)$ for all $d\ge2$. For
weighted vectors in $\R^2$, Liao--Shi--Solan--Tamam \cite[Theorem~1.1]{LSST}
proved $\dim_H\Sing(\wt)=2-(1+w_1)^{-1}$.

For unweighted Dirichlet-improvable vectors, Cheung--Chevallier
\cite[Corollaries~5.9, 6.12, and~7.13]{CC} showed that, as
$\eps\downarrow0$,
\[
 \dim_H\DI_{(1/2,1/2)}(\eps)-\frac43\asymp\sqrt\eps,
\]
whereas for every $d\ge3$ and all sufficiently small $\eps>0$,
\[
 c_d\eps
 \le\dim_H\DI_{(1/d,\ldots,1/d)}(\eps)-\frac{d^2}{d+1}
 \le C_d\sqrt\eps
\]
with constants $0<c_d\le C_d$. Thus the exponent of convergence of this
difference with respect to $\eps$ equals $1/2$ when $d=2$, and they
conjectured that for $d\ge3$ it exists and lies in $[(d-1)/d,1)$.

In the weighted planar case, Liao--Shi--Solan--Tamam
\cite[Theorem~1.3]{LSST} proved
\[
 \dim_H\DI_\wt(\eps)
 \le 2-\frac1{1+w_1}+O_\wt(\sqrt\eps),
\]
which Theorem~\ref{thm:main} complements with a matching lower bound.
For $d\ge3$, Theorem~\ref{thm:main} gives two-sided bounds of order
$\eps$ for all positive weights; in the unweighted case, this improves
the upper bound in \cite{CC} from $O_d(\sqrt\eps)$ to $O_d(\eps)$,
showing that the exponent of convergence equals $1$ and disproving the
conjecture in \cite{CC}. Over function fields, Aranov--Kim \cite{AK25}
computed the Hausdorff dimension of unweighted
singular vectors and bounded that of $\eps$-Dirichlet-improvable vectors.

For unweighted $m\times n$ systems of linear forms with $m+n\ge3$,
Kadyrov--Kleinbock--Lindenstrauss--Margulis \cite{KKLM17} proved upper
bounds for the Hausdorff dimensions of singular systems and systems
singular on average, and Das--Fishman--Simmons--Urba\'nski \cite{DFSU24}
obtained matching lower bounds. For one-parameter actions on finite-volume
homogeneous spaces, Guan--Shi \cite{GS20} proved that trajectories
diverging on average have positive Hausdorff codimension, settling a
conjecture of Cheung \cite{Cheung}. In the weighted setting,
Aggarwal--Ghosh \cite{AG26} obtained packing-dimension upper bounds for
singular matrices, Kleinbock--Moshchevitin--Warren--Weiss \cite{KMWW}
constructed totally irrational matrices satisfying uniform approximation
bounds for several weights, Datta--Tamam \cite{DT24} studied manifold
analogues, and Yang \cite{Yang26} showed that the subset of $\Sing(\wt)$
with arbitrarily slow divergence still has Hausdorff dimension
$2-(1+w_1)^{-1}$ when $d=2$ and $w_1>w_2$.

At the opposite extreme of uniform approximation, for $d\ge1$ and
positive weights $\wt$ with $\sum_{i=1}^d w_i=1$, let
$\operatorname{Bad}_\wt(\eps)$ be the set of $\xi\in\R^d$ satisfying
$\|q\xi-p\|_\wt>\eps/q$ for all sufficiently large $q\in\N$ and all
$p\in\Z^d$, and write $\operatorname{Bad}_d(\eps)$ in the unweighted
case. As $\eps\downarrow0$, Hensley \cite{Hen92} proved
$1-\dim_H\operatorname{Bad}_1(\eps)\sim6\eps/\pi^2$, and Simmons
\cite[Theorem~1.3]{Sim18} sharpened earlier estimates of
Broderick--Kleinbock \cite{BK15} to first-order asymptotics for badly
approximable matrices with respect to arbitrary fixed norms; in
particular,
\[
 d-\dim_H\operatorname{Bad}_d(\eps)
 \sim\frac{d\,2^d}{(d+1)\zeta(d+1)}\eps.
\]
For general weights, Kleinbock--Mirzadeh \cite[Theorem~1.4]{KM20}
established a lower bound of order $\eps/\log(1/\eps)$ on the Hausdorff
codimension of the subset where $\|q\xi-p\|_\wt\ge\eps/q$ holds for all
$q\in\N$ and $p\in\Z^d$, and conjectured in \cite[Section~9.1]{KM20} that
the logarithmic factor can be removed.

Several questions remain open. Theorem~\ref{thm:main} determines the order
of magnitude of $\dim_H\DI_\wt(\eps)-s_*$ as $\eps\downarrow0$, leaving
open whether $(\dim_H\DI_\wt(\eps)-s_*)/\sqrt\eps$ for $d=2$, or
$(\dim_H\DI_\wt(\eps)-s_*)/\eps$ for $d\ge3$, converges to a limit.
The corresponding first-order asymptotics of
$d-\dim_H\operatorname{Bad}_\wt(\eps)$ as $\eps\downarrow0$ for general
weights are likewise unknown. Another natural question is whether
$\eps\mapsto\dim_H\DI_\wt(\eps)$ and
$\eps\mapsto\dim_H\operatorname{Bad}_\wt(\eps)$ are continuous on
$(0,1)$: Moreira \cite[Theorem~1(i)]{Mor18} proved the continuity of
$\eps\mapsto\dim_H\operatorname{Bad}_1(\eps)$, and Simmons
\cite[Section~1.3]{Sim18} posed the corresponding question for badly
approximable unweighted matrices.

For $m\times n$ matrices,
Das--Fishman--Simmons--Urba\'nski \cite[Question~5.5 and Section~5.7]{DFSU24}
raised the problems of determining the dependence on $\eps$ of the
Hausdorff and packing dimensions of $\eps$-Dirichlet-improvable matrices
and computing the Hausdorff dimension of weighted singular matrices. For
$m+n\ge3$, Solan \cite[Corollaries~2.32 and~2.34]{Solan21} computed the
Hausdorff dimension of divergent trajectories in the full expanding
subgroup with respect to an expansion metric and deduced bounds in
standard metrics; passing to the Euclidean metric and restricting to the
smaller subgroup corresponding to weighted singular matrices are
separate problems (see \cite[Conjecture~3.2 and Question~3.3]{Solan21}).

\subsection*{Proof outline}
We use the accelerated self-similar coverings of Cheung \cite{Cheung},
Cheung--Chevallier \cite{CC}, and Liao--Shi--Solan--Tamam \cite{LSST},
with rectangles centered at rational points and a successor relation
on the integer approximation vectors. We modify the upper-bound covering
criteria in \cite[Theorem~3.1]{Cheung} and \cite[Lemma~2.8]{LSST}
by introducing an auxiliary function $p\ge1$, multiplying the contribution
of each successor $v$ of $u$ by $p(v)/p(u)$ (Lemma~\ref{lem:covering}).
For the lower bounds, we prove a mass-distribution criterion for
nested compact sets contained in rectangles
(Lemma~\ref{lem:rectangular-mass}).

For the upper bound when $d\ge3$, at each retained term $u$ of a
$\wt$-best-approximation sequence, we fix a rational hyperplane
$H_u$ through it and retain the first subsequent term $w$ outside $H_u$.
With $p\equiv1$, summing over intermediate approximations in $H_u$ and
over first exits from $H_u$ each introduces a factor of $(s-s_*)^{-1}$,
leading only to an $O_\wt(\sqrt\eps)$ bound. Instead, we define $p(u)$
from the first scale at which a nonzero rescaled dual vector has
$\ell^1$-norm at most $\eps$, and use \eqref{eq:DI} across intermediate
scales before exit to sample approximations in $H_u$ and bound the
intermediate sum by $O_\wt(p(u))$ (Lemma~\ref{lem:horizontal}). Partial
summation bounds the first-exit sum with the factor $p(w)$ omitted
(Lemma~\ref{lem:unweighted-exits}). For the remaining contribution of
$p(w)$, since $w_1+w_d<1$ when $d\ge3$, the dyadic sum over the dual
scale for exit normal vectors outside a bounded set of exceptional
directions converges without a second factor of $(s-s_*)^{-1}$
(Lemma~\ref{lem:ordinary-exits}), while exceptional exit normal vectors
are controlled uniformly in section rank by comparing $p(w)$ with $p(u)$
(Lemma~\ref{lem:exceptional-exits}). Combining these bounds gives
contraction in Lemma~\ref{lem:covering} at $s=s_*+C_\wt\eps$ for
$C_\wt$ sufficiently large and $\eps$ sufficiently small
(Proposition~\ref{prop:upper}).

For the lower bound when $d\ge3$, we construct nested compact sets in
rectangles whose side lengths make the parent approximation satisfy
\eqref{eq:DI} over a return duration $\tau$. Controlling only the
unrescaled lattice $\Lambda(\widehat u,t)$ bounds its last successive
minimum and supplies primitive candidates at the next return, but does
not separate them at the anisotropic scales of the child rectangles.
Following Liao--Shi--Solan--Tamam \cite[Section~4]{LSST}, we therefore
impose dual-lattice compactness both before and after a diagonal
rescaling adapted to the next return duration, which supplies enough
successors (Proposition~\ref{prop:successors}) and controls their
spatial distribution (see \eqref{eq:loweronestep}) uniformly in the
parent lattice. Since $1-w_1-w_2>0$, the relative contribution of
hyperplane sections of rank at most $d-2$ decays exponentially with
$\tau$; see \eqref{eq:lowerincidence}. Varying the next return duration over
$\asymp_\wt T$ values of order $T$ gives $\asymp_\wt\eps T e^{d\tau}$
children per vertex, and taking $T\asymp_\wt\eps^{-1}$ yields
$s_*+c_\wt\eps$ (Proposition~\ref{prop:lower}).

When $d=2$, we instead use a two-step construction: for each retained
approximation $u$, an integer lift of a $\|\cdot\|_\wt$-shortest nonzero
vector in the Farey lattice of $u$, together with $u$, spans a rational
plane in $\R^3$. Lemma~\ref{lem:planarplane} gives uniform approximation
along the corresponding affine line in $\R^2$ between the initial and
intermediate denominators. We choose an intermediate vector in this
plane and then a successor outside it. The two denominator ratios vary
independently over $\asymp_\wt T$ multiplicatively separated scales
(Lemma~\ref{lem:planar-successors}), yielding a factor of order $\eps T^2$
in the mass-distribution estimate (Lemma~\ref{lem:planar-distribution}).
Taking $T\asymp_\wt\eps^{-1/2}$ gives $s_*+c_\wt\sqrt\eps$
(Proposition~\ref{prop:planar-lower}).

\section{Preliminaries}\label{sec:prelim}

We write $A\ll B$ or $B\gg A$ for $A\le CB$, and $A\asymp B$
when both inequalities hold. Unless otherwise specified, the implied
constants depend only on $d$ and the fixed weights. By a dyadic scale we
mean a power of two ($2^k$ with $k\in\Z$), and sums over dyadic scales
run over powers of two in the indicated range unless stated otherwise.
For $S\subset\R^k$ and $\lambda>0$, write $\lambda S=\{\lambda x:x\in S\}$.
For a rank-$k$ lattice $L$ ($k\ge1$), let $\covol L$ denote the
$k$-dimensional Euclidean volume of a fundamental parallelepiped in
$\Span L$, with $\covol\{0\}=1$, and let
$L^*=\{b\in\Span L:b\cdot v\in\Z\text{ for all }v\in L\}$ denote its dual
in $\Span L$. A sublattice of a lattice is primitive if it is the
intersection of the ambient lattice with its real span.

\subsection{Best approximations and Farey lattices}
We use the weighted best approximations of
Liao--Shi--Solan--Tamam \cite[Section~5.1]{LSST} and the
Farey lattices of Cheung--Chevallier \cite{CC}. Let
\[
 \mathcal Q=\{u=(p,q)\in\Z^d\times\N:
                     \gcd(p_1,\ldots,p_d,q)=1\}.
\]
For $u=(p,q)\in\mathcal Q$, write $|u|=q$ and $\widehat u=p/q$.
\begin{definition}
For $\xi\in\R^d$ and $u\in\mathcal Q$, put
\[
 A(\xi,u)=\bigl\||u|(\xi-\widehat u)\bigr\|_\wt.
\]
We call $u$ a $\wt$-best approximation of $\xi$ if $A(\xi,u)<A(\xi,v)$
for every $v\in\mathcal Q$ with $|v|<|u|$, and
$A(\xi,u)\le A(\xi,v)$
for every $v\in\mathcal Q$ with $|v|=|u|$.
\end{definition}
For each $\xi\in\R^d$, choose one $\wt$-best approximation at each
denominator admitting one, resolving ties by a fixed rule. Enumerate these
approximations as $(u_n(\xi))$ (or simply $(u_n)$) in increasing order of
denominator. The sequence is infinite if $\xi\notin\mathbb Q^d$ and ends at
the primitive representation of $\xi$ otherwise.
Any two consecutive terms satisfy
\[
 A(\xi,u_n)=\min_{\substack{v\in\mathcal Q\\|v|<|u_{n+1}|}}A(\xi,v).
\]
Reducing an integer approximation to a primitive vector does not
increase its denominator or any coordinate error. Thus \eqref{eq:DI}
at $T=|u_{n+1}|$ implies $A(\xi,u_n)<\eps/|u_{n+1}|$.

For $u\in\mathcal Q$, let $\pi_u:\R^{d+1}\to\R^d$ be the
projection along $u$,
\[
 \pi_u(m,n)=m-n\widehat u,\quad\text{so that}\quad
 h_{-\widehat u}(m,n)=(\pi_u(m,n),n).
\]
\begin{definition}
The Farey lattice associated with $u\in\mathcal Q$ is
\[
 \Lambda_u=\pi_u(\Z^{d+1})=\Z^d+\Z\widehat u.
\]
\end{definition}
Since $u=(p,q)\in\mathcal Q$ is primitive, the class of $\widehat u$
modulo $\Z^d$ has order $q$. Hence $[\Lambda_u:\Z^d]=q$ and
$\covol\Lambda_u=|u|^{-1}$. By duality,
\[
 \Lambda_u^*=\{m\in\Z^d:m\cdot\widehat u\in\Z\}
 =\{m\in\Z^d:m\cdot p\equiv0\pmod q\},
\]
and $m\mapsto(m,-m\cdot\widehat u)$ is an isomorphism from $\Lambda_u^*$
onto $\Z^{d+1}\cap u^\perp$. Since $\Z^{d+1}\cap u^\perp$ is primitive in
$\Z^{d+1}$, a vector $m$ is primitive in $\Lambda_u^*$ if and only if
$(m,-m\cdot\widehat u)$ is primitive in $\Z^{d+1}$.
For $h>0$, put
\[
 D_h=\diag(h^{w_1},\ldots,h^{w_d})\quad\text{and}\quad
 B_\wt(h)=\prod_i[-h^{w_i},h^{w_i}].
\]
We also write
\[
 B_\wt(u,h)=\widehat u+|u|^{-1}B_\wt(h)
 =\{\xi\in\R^d:A(\xi,u)\le h\}.
\]
For $\eps>0$, define
\begin{equation}\label{eq:normalized-lattice}
 \Gamma_u(\eps)=D_{|u|/\eps}\Lambda_u,
 \quad\text{so that}\quad \covol\Gamma_u(\eps)=\eps^{-1},
\end{equation}
and its dual $\Gamma_u(\eps)^*=D_{\eps/|u|}\Lambda_u^*$ has covolume
$\eps$.
Put
\[
 r(u)=\min_{x\in\Lambda_u\setminus\{0\}}\|x\|_\wt.
\]
We use the inequality
$\|x+y\|_\wt\le 2^{1/w_d}\max(\|x\|_\wt,\|y\|_\wt)$.
Since $\|D_hx\|_\wt=h\|x\|_\wt$ for $h>0$,
\[
 \min_{x\in\Gamma_u(\eps)\setminus\{0\}}\|x\|_\wt
 =\frac{|u|r(u)}{\eps}.
\]

\begin{definition}
For $u\in\mathcal Q$, we define its domain of approximation by
\[
 \Delta(u)=\{\xi\in\R^d:u\text{ is a $\wt$-best approximation of }\xi\}.
\]
\end{definition}
The next lemma extends \cite[Lemma~5.4]{LSST} to arbitrary dimension;
see also \cite[Theorem~2.11]{Cheung} and \cite[Lemma~4.2]{CC}.

\begin{lemma}\label{lem:approximation-domain}
For every $u\in\mathcal Q$ with $|u|>1$,
\[
 B_\wt(u,2^{-1/w_d}r(u))
 \subseteq\Delta(u)\subseteq B_\wt(u,2^{1/w_d}r(u)).
\]
\end{lemma}
\begin{proof}
Since $|u|>1$, the class of $\widehat u$ modulo $\Z^d$ is nonzero.
Choosing its representative in $[-1/2,1/2]^d$ shows that $r(u)<1$.
Suppose $A(\xi,u)\le 2^{-1/w_d}r(u)$. For every
$v\in\mathcal Q\setminus\{u\}$ with $|v|\le|u|$, choose
$i\in\{1,\ldots,d\}$ such that $|\pi_u(v)_i|\ge r(u)^{w_i}$.
Then
\[
 |v|\,|\xi_i-\widehat v_i|
 \ge r(u)^{w_i}-\frac{|v|}{|u|}A(\xi,u)^{w_i}
 \ge\left(1-\frac{|v|}{2|u|}\right)r(u)^{w_i}
 \ge A(\xi,u)^{w_i}.
\]
The final inequality is strict when $|v|<|u|$, proving that $u$ is a
$\wt$-best approximation of $\xi$.

For the reverse inclusion, take $x\in\Lambda_u$ with
$\|x\|_\wt=r(u)$. Since $r(u)<1$, we have $x\notin\Z^d$.
Take an integer lift $v$ of $x$ with last coordinate in
$\{1,\ldots,|u|-1\}$. Replacing $v$ by $u-v$ if its last coordinate
is greater than $|u|/2$, we obtain $0<|v|\le|u|/2$ and
$\|\pi_u(v)\|_\wt=r(u)$. Minimality of $r(u)$ implies that
$x$ is primitive in $\Lambda_u$, and hence $v$ is primitive in
$\Z^{d+1}$.
If $\xi\in\Delta(u)$ and $A(\xi,u)>2^{1/w_d}r(u)$, then for all
$i=1,\ldots,d$,
\[
 |v|\,|\xi_i-\widehat v_i|
 \le r(u)^{w_i}+\tfrac12 A(\xi,u)^{w_i}
 <A(\xi,u)^{w_i},
\]
contrary to the best-approximation property of $u$.
\end{proof}

\begin{lemma}\label{lem:integer-pairing}
Let $u,v\in\mathcal Q$ be consecutive $\wt$-best approximations
of $\xi$, with $|u|<|v|$, and suppose
$A(\xi,u)<\eps/|v|$. Then
\[
 \bigl\|D_{|v|/\eps}\pi_v(u)\bigr\|_\infty<2
 \quad\text{and}\quad |v|r(v)\le 2^{1/w_d}\eps.
\]
Every $(m,n)\in\Z^{d+1}$ orthogonal to $v$
with $\|D_{\eps/|v|}m\|_1<1/2$ is also orthogonal to $u$.
\end{lemma}
\begin{proof}
Since $A(\xi,v)<A(\xi,u)$ and $|u|<|v|$, for every
$i=1,\ldots,d$ we have
\[
 |\pi_v(u)_i|
 \le A(\xi,u)^{w_i}+\frac{|u|}{|v|}A(\xi,v)^{w_i}
 <2(\eps/|v|)^{w_i}.
\]
Since $r(v)\le\|\pi_v(u)\|_\wt$, this also gives
$|v|r(v)\le 2^{1/w_d}\eps$.
Since $(m,n)\cdot v=0$, we have
$(m,n)\cdot u=m\cdot\pi_v(u)\in\Z$. This integer vanishes because
\[
 |(m,n)\cdot u|
 \le\sum_i|m_i|\,|\pi_v(u)_i|
 <2\sum_i|m_i|(\eps/|v|)^{w_i}
 =2\|D_{\eps/|v|}m\|_1<1.\qedhere
\]
\end{proof}

\subsection{Lattice estimates}
For a rank-$k$ lattice $L$ ($k\ge1$), write $\lambda_i(L)$ for the $i$th
successive minimum of $L$ with respect to the Euclidean norm on $\Span L$.
For an origin-symmetric convex body $K$ with nonempty interior in
$\Span L$, write $\lambda_i(K,L)$ for its $i$th successive minimum with
respect to $K$.

\begin{lemma}\label{lem:lattice}
Let $L$ have rank $k$.
\begin{enumerate}
\item Let $K$ be an origin-symmetric convex body with nonempty
interior in $\Span L$. If $L\cap K$ spans $\Span L$, then
\[
 \#(L\cap K)\ll_k \frac{\vol_k K}{\covol L}.
\]
If $\lambda_1(K,L)\ge1$, then $\vol_k K/\covol L\ll_k1$ and, for
$A\ge1$,
\[
 \#(L\cap AK)\ll_k A^{k-1}
                 +A^k\frac{\vol_k K}{\covol L}.
\]
The same estimate holds for translates of $AK$.
\item For a primitive sublattice $U\subset L^*$,
\begin{equation}\label{eq:orthogonalcovol}
 \covol U=\frac{\covol(L\cap U^\perp)}{\covol L}.
\end{equation}
In particular, for primitive $b\in L^*$,
\[
 \covol(L\cap b^\perp)=\|b\|_2\,\covol L.
\]
\item For $k\ge2$, $K=\prod_{i=1}^k[-R_i,R_i]$ with $R_1,\ldots,R_k>0$,
and every $b\in\R^k\setminus\{0\}$,
\[
 \frac{\vol_{k-1}(K\cap b^\perp)}{\|b\|_2}
 \asymp_k \frac{\vol_k K}{\max_i R_i|b_i|}.
\]
For $1\le j\le k$ and every $j$-dimensional subspace $V\subset\R^k$,
\begin{equation}\label{eq:box-section-volume}
 \vol_j(K\cap V)\ll_k
 \max_{\substack{I\subseteq\{1,\ldots,k\}\\|I|=j}}
       \prod_{i\in I}R_i.
\end{equation}
\end{enumerate}
\end{lemma}
\begin{proof}
The first bound in (1) is \cite[Lemma~3.3]{LSST} applied in $\Span L$.
If $\lambda_1(K,L)\ge1$, Minkowski's second theorem gives
$\vol_k K/\covol L\ll_k1$. Put $r=\dim\Span(L\cap AK)$.
If $r=k$, apply the first bound to $AK$. If $r=0$, then
$L\cap AK=\{0\}$. If $1\le r<k$, the translates of
$\frac12(K\cap\Span(L\cap AK))$ centered at the lattice points of
$AK$ have disjoint interiors in $\Span(L\cap AK)$ and lie in
$(A+\frac12)(K\cap\Span(L\cap AK))$; hence their number is at most
$(2A+1)^r\ll_k A^{k-1}$. A nonempty translate $z+AK$ may be
shifted by one of its lattice points into $2AK$.

For (2), extend a basis $b_1,\ldots,b_j$ of $U$ to a basis of
$L^*$, and denote the dual basis of $L$ by $v_1,\ldots,v_k$.
Then $L\cap U^\perp$ is generated by $v_{j+1},\ldots,v_k$,
whereas the orthogonal projections of $v_1,\ldots,v_j$ onto
$\Span U$ form the basis dual to $b_1,\ldots,b_j$.
Thus $\covol(\operatorname{proj}_{\Span U}L)=1/\covol U$;
factoring $\covol L$ along $U^\perp$ and $\Span U$
gives \eqref{eq:orthogonalcovol}.

The hyperplane-section estimate in (3) follows from
\cite[Lemma~3.6]{LSST}, since
$\sup_{x\in K}|b\cdot x|=\sum_iR_i|b_i|\asymp_k\max_iR_i|b_i|$.
To prove \eqref{eq:box-section-volume}, let $A$ be the $k\times j$
matrix whose columns form an orthonormal basis of $V$. For a set
$I\subseteq\{1,\ldots,k\}$ of size $j$, let $A_I$ denote the square
submatrix formed by the rows indexed by $I$. By Cauchy--Binet,
\[
 1=\det(A^{\mathsf T}A)=\sum_{|I|=j}|\det A_I|^2.
\]
Thus $|\det A_I|\ge\binom{k}{j}^{-1/2}$ for some $I$.
Write $\pi_I$ for the coordinate projection onto the coordinates
indexed by $I$. Its restriction to $V$ is invertible and has
Jacobian $|\det A_I|$. Since
$\pi_I(K\cap V)\subseteq\prod_{i\in I}[-R_i,R_i]$,
\[
 \vol_j(K\cap V)
 =\frac{\vol_j(\pi_I(K\cap V))}{|\det A_I|}
 \le 2^j\binom{k}{j}^{1/2}\prod_{i\in I}R_i,
\]
which proves \eqref{eq:box-section-volume}.
\end{proof}

For regions not necessarily symmetric about the origin, we use M\"obius
inversion as in \cite[Lemma~3.1]{LSST}, together with the
Lipschitz-boundary estimate of Widmer \cite[Theorem~5.4]{Widmer10}.
\begin{lemma}\label{lem:primitive-count}
Let $k\ge2$ and $B\ge1$. Let $\Omega\subset\R^k$ be contained
in the Euclidean ball of radius $B$ about the origin, and suppose
that $\partial\Omega$ is covered by $M$ Lipschitz images of
$[0,1]^{k-1}$ with Lipschitz constants at most $B$.
If $L\subset\R^k$ is a full-rank lattice with $\lambda_k(L)\le1$, then
\[
 \#\{x\in L\cap\Omega:x\text{ primitive}\}
 =\frac1{\covol L}\left[
 \frac{\vol_k\Omega}{\zeta(k)}
 +O_{k,B,M}\!\left(\lambda_k(L)
             \bigl(1+\log(1/\lambda_1(L))\bigr)\right)\right].
\]
\end{lemma}
\begin{proof}
By \cite[Theorem~5.4]{Widmer10}, for every integer $j\ge1$,
\[
 \#\bigl((L\setminus\{0\})\cap j^{-1}\Omega\bigr)
 =\frac{\vol_k\Omega}{j^k\covol L}
 +O_{k,B,M}\!\left(\sum_{i=0}^{k-1}
        \frac{j^{-i}}{\lambda_1(L)\cdots\lambda_i(L)}\right),
\]
where the empty product is one; omitting the origin is absorbed
by the term $i=0$. By M\"obius inversion, the primitive count is
the sum of these expressions with coefficients $\mu(j)$. The sum
stops at $j\le B/\lambda_1(L)$, since $j^{-1}\Omega$ contains no
nonzero lattice point for larger $j$. Put $J=B/\lambda_1(L)$.
The sums $\sum_{j\le J}j^{-i}$ are $O(J)$ for $i=0$,
$O(1+\log J)$ for $i=1$, and $O(1)$ for $i\ge2$.
Since $\lambda_i(L)\le1$ for all $i$ and $B$ is fixed, the summed
errors are at most
$O_{k,B,M}((1+\log(1/\lambda_1(L)))/
(\lambda_1(L)\cdots\lambda_{k-1}(L)))$.
The tail of $\sum_{j\ge1}\mu(j)j^{-k}=1/\zeta(k)$ contributes
$O_{k,B}(\lambda_1(L)^{k-1}/\covol L)$. Consequently,
\[
 \#\{x\in L\cap\Omega:x\text{ primitive}\}
 =\frac{\vol_k\Omega}{\zeta(k)\covol L}
 +O_{k,B,M}\!\left(
  \frac{1+\log(1/\lambda_1(L))}{\lambda_1(L)\cdots\lambda_{k-1}(L)}
  +\frac{\lambda_1(L)^{k-1}}{\covol L}\right).
\]
Minkowski's second theorem and $\lambda_1(L)^{k-1}\le\lambda_k(L)$
complete the proof.
\end{proof}

\subsection{Covering and mass distribution}
Let $J$ be an index set and $\sigma\subset J\times J$ a relation.
For $u\in J$, write $\sigma(u)=\{v\in J:(u,v)\in\sigma\}$.
A $\sigma$-admissible path is a sequence $(u_j)$ satisfying
$u_{j+1}\in\sigma(u_j)$. We write
$B_\infty(\xi,r)=\xi+[-r,r]^d$, and for nonempty sets
$E,F\subset\R^d$ let
$\dist_\infty(E,F)=\inf\{\|x-y\|_\infty:x\in E,\ y\in F\}$.
Also put $\log^+x=\max\{0,\log x\}$ for $x>0$.

The following lemmas use the self-similar covering and
mass-distribution methods of \cite[Theorems~3.1--3.2]{Cheung} and
\cite[Section~3]{CC}, together with the rectangular estimates of
\cite[Theorem~2.1 and Lemma~2.8]{LSST}.
\begin{lemma}\label{lem:covering}
Let $S\subset\R^d$, $0<\eps\le1$, and $d-1<s<d$, and put
$\delta=(1+w_1)(s-s_*)$.
Let $\sigma$ be a relation on $\mathcal Q$ such that
$|v|>|u|$ whenever $v\in\sigma(u)$.
Suppose every $\xi\in S$ admits a $\sigma$-admissible path $(u_j)$
with $\xi\in B_\wt(u_j,\eps/|u_j|)$ for all sufficiently large $j$.
If there exist $p:\mathcal Q\to[1,\infty)$ and $\lambda\in(0,1)$ such that
\begin{equation}\label{eq:contraction}
 \sum_{v\in\sigma(u)}
 \left(\frac{|u|}{|v|}\right)^{d+\delta}p(v)
 \le\lambda p(u)
\end{equation}
for every $u\in\mathcal Q$, then $\dim_H S\le s$.
\end{lemma}
\begin{proof}
An axis-parallel rectangle with ordered half-sides $l_1\le\cdots\le l_d$ can
be covered by $O_d(\prod_{i=2}^d l_i/l_1)$ cubes of side $2l_1$.
The sum of the $s$th powers of their diameters is therefore at most
$C_dl_1^{s-d+1}\prod_{i=2}^d l_i$.
For $B_\wt(u,\eps/|u|)$, the half-sides are
$\eps^{w_i}|u|^{-1-w_i}$, in increasing order. Thus this
rectangle has a finite cover $\mathcal C(u)$ by cubes of diameter
$2\sqrt d\,\eps^{w_1}|u|^{-1-w_1}$ such that
\[
 \sum_{U\in\mathcal C(u)}(\operatorname{diam} U)^s
 \le C\eps^{\,1-w_1(d-s)}|u|^{-d-\delta}.
\]

For $u_0\in\mathcal Q$, let $S(u_0)$ be the set of $\xi$ admitting a
$\sigma$-admissible path starting at $u_0$ such that
$\xi\in B_\wt(u_j,\eps/|u_j|)$ for all $j\ge0$.
Iterating \eqref{eq:contraction} gives
\[
 \sum_{u_1\in\sigma(u_0)}\cdots\sum_{u_j\in\sigma(u_{j-1})}
 |u_j|^{-d-\delta}p(u_j)
 \le\lambda^j|u_0|^{-d-\delta}p(u_0).
\]

For each $j$, the union of $\mathcal C(u_j)$ over length-$j$
admissible paths from $u_0$ covers $S(u_0)$.
Since $|u_j|\ge|u_0|+j$, all cubes in this cover have diameter at most
$2\sqrt d\,\eps^{w_1}(|u_0|+j)^{-1-w_1}$, which tends to zero.
By the iterated contraction estimate and $p(u_j)\ge1$, the sum
of the $s$th powers of the diameters is at most
\[
 C\eps^{1-w_1(d-s)}\lambda^j|u_0|^{-d-\delta}p(u_0).
\]
Hence $\mathcal H^s(S(u_0))=0$. Discarding finitely many initial
terms shows that $S\subseteq\bigcup_{u_0\in\mathcal Q}S(u_0)$.
Countability of $\mathcal Q$ now gives $\mathcal H^s(S)=0$.
\end{proof}

\begin{lemma}\label{lem:rectangular-mass}
Let $\sigma$ define a rooted tree with root $u_0$ and finite branching.
Associate to each vertex $u$ a nonempty compact set $B(u)\subset\R^d$
contained in an axis-parallel rectangle with ordered positive half-sides
$l_1(u)\le\cdots\le l_d(u)$.
Suppose there are $a,\rho>0$ such that $B(v)\subset B(u)$ and
$l_i(v)\le e^{-a}l_i(u)$ for every $v\in\sigma(u)$ and all $i$,
and the sets associated with distinct children of $u$ have
sup-norm distance at least $\rho l_1(u)$.

For each vertex $u$, let $(P_u(v))_{v\in\sigma(u)}$ be a
probability vector with positive entries. Suppose that for some
$\lambda\ge1$ and all $u$, $\xi\in\R^d$, and $r\ge l_1(u)$, we have
\begin{equation}\label{eq:tree-one-step}
 \sum_{\substack{v\in\sigma(u)\\
B(v)\cap B_\infty(\xi,r)\ne\varnothing}}
P_u(v)\prod_{i=2}^d\min(1,r/l_i(v))
 \le \lambda\prod_{i=2}^d\min(1,r/l_i(u)).
\end{equation}

Let $d-1<s<d$. Suppose that there is $C_s>0$, independent of the
path and its length, such that every finite path $u_0,\ldots,u_j$
satisfies
\begin{equation}\label{eq:tree-rectangular-mass}
 \left(\prod_{k=0}^{j-1}P_{u_k}(u_{k+1})\right)
 \lambda^{\,1+a^{-1}\log(l_d(u_j)/l_1(u_j))}
 \le C_s l_1(u_j)^{s-d+1}\prod_{i=2}^d l_i(u_j).
\end{equation}
Then the limit set $\bigcup_{(u_j)}\bigcap_{j\ge0}B(u_j)$ has
Hausdorff dimension at least $s$, where the union runs over all
infinite $\sigma$-admissible paths from $u_0$.
\end{lemma}
\begin{proof}
The nesting and uniform contraction imply that each infinite
path determines a unique point. The transition probabilities
induce a probability measure $\mu$ on the limit set.
For a vertex $u$ on a path $u_0,\ldots,u_j=u$, let $[u]$ denote the
set of limit points of paths passing through $u$. Then
$\mu([u])=\prod_{k=0}^{j-1}P_{u_k}(u_{k+1})$.
Write $\mu_u(A)=\mu(A\cap[u])/\mu([u])$ for the conditional measure.

Fix $u$, $\xi\in\R^d$, and $r\ge l_1(u)$, and put
$n=\lceil a^{-1}\log^+(l_d(u)/r)\rceil$.
Every descendant $v$ at depth $n$ satisfies $l_d(v)\le r$, hence
$\mu_v(B_\infty(\xi,r))\le1=\prod_{i=2}^d\min(1,r/l_i(v))$.
Apply \eqref{eq:tree-one-step} successively from depth $n$ back to
$u$, using $\mu_v=\sum_{w\in\sigma(v)}P_v(w)\mu_w$ and the fact
that only children meeting $B_\infty(\xi,r)$ contribute.
Since $r\ge l_1(v)$ at every descendant, the iteration contributes
at most the factor
$\lambda^n\le\lambda^{1+a^{-1}\log^+(l_d(u)/r)}$. Therefore
\[
 \mu_u(B_\infty(\xi,r))
 \le\lambda^{\,1+a^{-1}\log^+(l_d(u)/r)}
       \prod_{i=2}^d\min(1,r/l_i(u)).
\]

Since $r\ge l_1(u)$, we have
$\log^+(l_d(u)/r)\le\log(l_d(u)/l_1(u))$.
Multiplying by $\mu([u])$ and applying
\eqref{eq:tree-rectangular-mass} therefore gives
\[
 \mu([u]\cap B_\infty(\xi,r))
 \le C_s l_1(u)^{s-d+1}\prod_{i=2}^d\min(l_i(u),r)
 \le C_sr^s.
\]

For $0<r<l_1(u_0)$, stop each path at the first vertex $u$ with
$l_1(u)\le r$; such a vertex exists by uniform contraction.
The resulting stopping cylinders partition the limit set.
Two distinct stopping paths diverge at a vertex
whose shortest half-side is greater than $r$, so the separation assumption
places their cylinders at sup-norm distance at least $\rho r$.
By packing, at most $O_{d,\rho}(1)$ stopping cylinders meet
$B_\infty(\xi,r)$. Since $l_1(u)\le r$ at each stopping vertex,
the preceding mass bound gives
$\mu(B_\infty(\xi,r))\ll_{d,\rho}C_sr^s$.
The mass-distribution principle proves the result.
\end{proof}

\section{The upper bound in higher dimensions}\label{sec:upper}

We prove the upper bound in Theorem~\ref{thm:main} for $d\ge3$.

\begin{proposition}\label{prop:upper}
Let $d\ge3$ and let $\wt$ satisfy \eqref{eq:weights}. Then
\[
 \dim_H\DI_\wt(\eps)\le s_*+C_\wt\eps
\]
for all sufficiently small $\eps>0$.
\end{proposition}

For each best approximation $u$, we choose a normal vector
$u^*\in\Z^{d+1}$ and follow the best approximations in
$(u^*)^\perp$ up to the first exit $w$. We count these exits together
with their immediate predecessors $v$, using dual coordinates at $v$.

Fix $0<\eps<1$ sufficiently small, and set $r=1-w_d/2$. Since
$d\ge3$, we have $1-w_1\ge2w_d$, and hence
$w_1\le1-2w_d\le r(1-w_d)$. Thus
\begin{equation}\label{eq:upper-r}
 r=1-\frac{w_d}{2}\in(1-w_d,1),\quad r>w_1,\quad\text{and}\quad
 1-\frac{w_1}{r}\ge w_d.
\end{equation}

\subsection{Choice of normal vector}
\begin{definition}
For $c\in\R^d\setminus\{0\}$, define
$\tau(c)=\min\{t>0:\|D_{t^{-1}}c\|_1\le\eps\}$, so that $\tau(c)>0$ is
uniquely determined by
\begin{equation}\label{eq:dualtime}
 \sum_{i=1}^d|c_i|\tau(c)^{-w_i}=\eps.
\end{equation}
For each $T>0$, the condition $\tau(c)\le T$ implies
$|c_i|\le\eps T^{w_i}$ for every $i$. Since $\Gamma_u(\eps)^*$ is
discrete, $\tau$ attains its minimum on its nonzero vectors. Define
\[
 \tau_*(u)=\min_{0\ne c\in\Gamma_u(\eps)^*}\tau(c)
 \quad\text{and}\quad F(u)=\tau_*(u)^{-r}.
\]
\end{definition}

If $F(u)\ge\eps$, then any vector $c\in\Gamma_u(\eps)^*\setminus\{0\}$
minimizing $\tau$ satisfies $\tau(c)=\tau_*(u)\le\eps^{-1/r}$, and
\eqref{eq:dualtime} together with \eqref{eq:upper-r} gives
$\|c\|_1\le\eps^{1-w_1/r}\le\eps^{w_d}$.
When $F(u)<\eps$, Minkowski's first theorem and
$\covol\Gamma_u(\eps)^*=\eps$ from
\eqref{eq:normalized-lattice} show that a nonzero vector of minimum
$\ell^1$-norm in $\Gamma_u(\eps)^*$ satisfies
$\|c\|_1\ll_d\eps^{1/d}\le\eps^{w_d}$. Since $\tau(kc)>\tau(c)$ and $\|kc\|_1>\|c\|_1$ for
$|k|\ge2$, every nonzero minimizer of either function is primitive
in $\Gamma_u(\eps)^*$.

Accordingly, for each $u\in\mathcal Q$ choose
$c_u\in\Gamma_u(\eps)^*\setminus\{0\}$ to minimize $\tau$ when
$F(u)\ge\eps$, and $\|\cdot\|_1$ otherwise. This vector is primitive
in $\Gamma_u(\eps)^*$ and satisfies
\begin{equation}\label{eq:normalbound}
 \|c_u\|_1\ll_\wt\eps^{w_d}.
\end{equation}
Set $m_u=D_{|u|/\eps}c_u\in\Lambda_u^*$ and
$u^*=(m_u,-m_u\cdot\widehat u)\in\Z^{d+1}$, and put
$H_u=(u^*)^\perp$. These vectors are primitive in their respective
lattices.

Suppose $u$ is a term of $(u_n(\xi))$ and \eqref{eq:DI} holds at
$T=|u|$. Then $|u|>1$, the preceding term $u'$ satisfies
$A(\xi,u)\le A(\xi,u')<\eps/|u|$, and \eqref{eq:normalbound} gives
$\|c_u\|_1<1/2$ for sufficiently small $\eps$; hence
Lemma~\ref{lem:integer-pairing} yields
\begin{equation}\label{eq:best-node-properties}
 |u|r(u)\le2^{1/w_d}\eps,\quad
 \xi\in B_\wt(u,\eps/|u|),\quad\text{and}\quad u'\in H_u.
\end{equation}

\begin{definition}\label{def:first-exit}
For $u\in\mathcal Q$, let $\Xi_\eps(u)$ consist of those
$\xi\in\Delta(u)$ such that \eqref{eq:DI} has a solution for every
$T\ge|u|$, and the chosen best-approximation sequence contains
$u=u_{i_0}(\xi)$ for some $i_0\ge1$ and subsequently leaves $H_u$.
Define
\[
 \nu_u(\xi)=\min\bigl\{n>i_0:u_n(\xi)\notin H_u\bigr\}
 \quad\text{and}\quad
 \sigma_\eps(u)=\bigl\{u_{\nu_u(\xi)}(\xi):\xi\in\Xi_\eps(u)\bigr\}.
\]
If $v=u_{\nu_u(\xi)-1}(\xi)$ immediately precedes
$w=u_{\nu_u(\xi)}(\xi)$ for some $\xi\in\Xi_\eps(u)$, then
\[
 v\in H_u,\quad w\notin H_u,\quad\text{and}\quad |u|\le|v|<|w|.
\]
\end{definition}

For $0<\delta\le1$, put
\begin{equation}\label{eq:stateweight}
 p(u)=1+\delta^{-1}\min(1,F(u)).
\end{equation}

\subsection{Approximations in a hyperplane}
Fix $u\in\mathcal Q$ and write $q=|u|$. The intermediate best
approximations in $H_u$ are
\[
 \mathcal D_{u^*}(u)=\bigl\{u_k(\xi):\xi\in\Xi_\eps(u),\ 
 u_{i_0}(\xi)=u,\ i_0<k<\nu_u(\xi)\bigr\}.
\]
For $T>0$, set
\[
 \Phi(T)=\frac1\eps\sum_{i=1}^d|c_{u,i}|(q/T)^{w_i}
 =\frac1\eps\sum_{i=1}^d|m_{u,i}|(\eps/T)^{w_i},
 \quad\text{so that}\quad \Phi(q\tau(c_u))=1.
\]
The function $\Phi(T)$ is decreasing, while
$T\Phi(T)=\eps^{-1}\sum_{i=1}^d|c_{u,i}|q^{w_i}T^{1-w_i}$ is increasing,
since $0<w_i<1$.

Suppose $\sigma_\eps(u)\ne\varnothing$, and take $\xi\in\Xi_\eps(u)$.
Since $\xi$ satisfies \eqref{eq:DI} at $T=q$, we have $q>1$, and
\eqref{eq:best-node-properties} gives $qr(u)\le2^{1/w_d}\eps$. Put
\[
 k_u=\frac{\vol_{d-1}(B_\wt(r(u))\cap m_u^\perp)}
          {\covol(\Lambda_u\cap m_u^\perp)}.
\]
Since $m_u$ is primitive in $\Lambda_u^*$, Lemma~\ref{lem:lattice}(2)
gives $\covol(\Lambda_u\cap m_u^\perp)=q^{-1}\|m_u\|_2$.
Moreover, $\vol_d B_\wt(r(u))=2^d r(u)$ and
\[
 \max_{1\le i\le d} r(u)^{w_i}|m_{u,i}|
 \asymp_d\sum_{i=1}^d r(u)^{w_i}|m_{u,i}|
 =\eps\Phi(\eps/r(u)).
\]
Since $\eps/r(u)\ge2^{-1/w_d}q$, monotonicity of $T\Phi(T)$ and $\Phi(T)$
gives
\[
 \frac{\eps}{r(u)}\Phi(\eps/r(u))
 \ge(2^{-1/w_d}q)\Phi(2^{-1/w_d}q)
 \ge2^{-1/w_d}q\Phi(q),
\]
while $\lambda_1(B_\wt(r(u))\cap m_u^\perp,\Lambda_u\cap m_u^\perp)\ge1$
by definition of $r(u)$. Lemma~\ref{lem:lattice}(1) and~(3) therefore
give
\begin{equation}\label{eq:compact-density}
 k_u\asymp_d\frac{q}{(\eps/r(u))\Phi(\eps/r(u))}
 \ll_\wt\frac1{\Phi(q)}\quad\text{and}\quad k_u\ll_d1.
\end{equation}
Since $\Phi(q)=\eps^{-1}\|c_u\|_1$, this also implies
$k_u\|c_u\|_1\ll_\wt\eps$.
If $A\ge1$ and $v\in H_u\cap\mathcal Q$ satisfies
$Aq\le|v|<2Aq$ and $A(\xi,v)\ll_\wt r(u)$ for some
$\xi\in\Delta(u)$, then $\pi_u(v)\in CA\,B_\wt(r(u))$ for some
$C=C_\wt>0$.
Each projection has $O(A)$ integer lifts with denominator in $[Aq,2Aq)$,
so Lemma~\ref{lem:lattice}(1) gives
\begin{equation}\label{eq:hyperplane-count}
 \#\bigl\{v\in H_u\cap\mathcal Q:Aq\le|v|<2Aq,
                 \ \pi_u(v)\in CA\,B_\wt(r(u))\bigr\}
 \ll_\wt A^{d-1}+k_uA^d\quad\text{for }A\ge1.
\end{equation}

\begin{lemma}\label{lem:horizontal}
For every $u\in\mathcal Q$ and $0<\delta\le1$,
\[
 \sum_{v\in\mathcal D_{u^*}(u)}
 \left(\frac{|u|}{|v|}\right)^{d+\delta}\ll_\wt p(u).
\]
\end{lemma}
\begin{proof}
If $\mathcal D_{u^*}(u)=\varnothing$, the conclusion is immediate.
Otherwise $\sigma_\eps(u)\ne\varnothing$, so \eqref{eq:compact-density}
and \eqref{eq:hyperplane-count} apply; in particular, summing
\eqref{eq:hyperplane-count} over dyadic $A\ge1$ bounds the desired sum by
$O_\wt(1+k_u/\delta)$.

Set $C_*=2^{1+1/w_d}$, and let $L_*>2$ be a sufficiently large
constant depending only on $d$ and $\wt$, to be fixed below.
Put $T_0=C_*\eps/r(u)$ and choose $T_*$ so that
$\Phi(T_*)=L_*$. By \eqref{eq:best-node-properties}, $T_0\ge q$.
If $\Phi(q)\le L_*$, then $1=\Phi(q\tau(c_u))\le\tau(c_u)^{-w_d}\Phi(q)$
whenever $\tau(c_u)\ge1$, so $\tau_*(u)\le\tau(c_u)\le\max(1,L_*^{1/w_d})$.
Hence $\min(1,F(u))\gg_{\wt,L_*}1$, and the global bound
$O_\wt(1+k_u/\delta)$ gives the result in this case.
Suppose $\Phi(q)>L_*$ and write $A_*=T_*/q>1$.
Since $\Phi(T_*)=L_*$,
\begin{equation}\label{eq:sample-scale-comparison}
 L_*A_*^{w_d}\le\Phi(q)\le L_*A_*^{w_1}
 \quad\text{and}\quad A_*\asymp_{\wt,L_*}\tau(c_u).
\end{equation}
Since $\Phi(q\tau(c_u))=1$, $L_*=\Phi(T_*)$ is a convex
combination of $(\tau(c_u)/A_*)^{w_i}$ ($1\le i\le d$), which bounds
$\tau(c_u)/A_*$ above and below in terms of $L_*$ and $\wt$.
If $T_0\ge T_*$, then $\eps/r(u)=C_*^{-1}T_0\ge C_*^{-1}T_*$; since
$C_*>1$, monotonicity of $T\Phi(T)$ and $\Phi(T)$ together with
\eqref{eq:compact-density} gives
\[
 k_u\ll_d\frac{q}{(C_*^{-1}T_*)\Phi(C_*^{-1}T_*)}
 \le\frac{C_*q}{T_*\Phi(T_*)}\ll_\wt A_*^{-1}.
\]
Since $\Phi(q)>1=\Phi(q\tau(c_u))$, we have $\tau(c_u)>1$; as $r<1$ and
$\tau(c_u)\ge\tau_*(u)$, the global count gives
$O_\wt(1+\delta^{-1}A_*^{-1})=O_{\wt,L_*}(1+\delta^{-1}\tau(c_u)^{-r})
=O_{\wt,L_*}(p(u))$.

Assume henceforth that $T_0<T_*$. By \eqref{eq:hyperplane-count},
the contribution of approximations with $|v|\le T_*$ is
\[
 O_\wt\bigl(1+k_u(1+\log A_*)\bigr)=O_\wt(1),
\]
since $k_u\ll_\wt\Phi(q)^{-1}$ and $A_*^{w_d}\le\Phi(q)/L_*<\Phi(q)$,
with an implied constant independent of $L_*$. We control the later
approximations by sampling between $T_0$ and $T_*$.

Fix $\xi\in\Delta(u)$ and $v\in\mathcal Q$ with $|v|<q$.
Since $\pi_u(v)\ne0$, there is an $i$ such that
$|\pi_u(v)_i|\ge r(u)^{w_i}$. The best-approximation property of $u$
then gives
\[
 r(u)^{w_i}\le|\pi_u(v)_i|
 \le A(\xi,v)^{w_i}+\frac{|v|}{q}A(\xi,u)^{w_i}
 <2A(\xi,v)^{w_i}.
\]
Since $w_i\ge w_d$, for every $T\ge T_0$ we have
\[
 A(\xi,v)>2^{-1/w_d}r(u)
 =\frac{2\eps}{T_0}\ge\frac{2\eps}{T}>\frac{\eps}{T}.
\]
Consequently, any primitive vector satisfying $A(\xi,v)\le\eps/T$
has denominator at least $q$.
For each $v\in\mathcal D_{u^*}(u)$ with $|v|\ge T_*$,
choose $\xi_v\in\Xi_\eps(u)$ from the definition of
$\mathcal D_{u^*}(u)$, and let $\Xi$ be the set of these points $\xi_v$.
For $T_0\le T\le T_*$, put
\[
 \mathcal C(T)=\{u'\in H_u\cap\mathcal Q: q\le|u'|<T,
          \ A(\xi,u')\le\eps/T\text{ for some }\xi\in\Xi\}.
\]
For each $\xi=\xi_v\in\Xi$ and $T\in[T_0,T_*]$,
\eqref{eq:DI} supplies a chosen best approximation $u'$ with $|u'|<T$ and
$A(\xi,u')\le\eps/T$. The preceding exclusion gives $|u'|\ge q$;
since $|v|\ge T_*$ and $v$ precedes the first exit, $u'\in H_u$.
The set $\mathcal C(T)$ is finite since $|u'|<T$ and
$\Delta(u)$ is bounded.
For each $\xi\in\Xi$, choose $u'\in\mathcal C(T)$ minimizing $|u'|$
subject to $A(\xi,u')\le\eps/T$, resolving ties by a fixed rule.

For such a pair $(\xi,u')$, suppose
$0\ne z=\pi_{u'}(v)\in\frac14 B_\wt(\eps/T)$ for some
$v\in H_u\cap\Z^{d+1}$. Since $u'$ is primitive and
$\ker(\pi_{u'}|_{\Z^{d+1}})=\Z u'$, we may take the last coordinate $n$
of $v$ with $|n|\le|u'|/2$. As $\eps/T<1$, $n=0$ would give $0\ne z\in\Z^d\cap(-1,1)^d$.
Changing signs, assume $0<n\le|u'|/2$. For each
$1\le i\le d$, we have
\[
 |n\xi_i-v_i|
 \le |z_i|+\frac n{|u'|}\bigl||u'|(\xi_i-\widehat{u'}_i)\bigr|
 \le\frac34(\eps/T)^{w_i}< (\eps/T)^{w_i}.
\]
The primitive reduction $v'$ satisfies
$q\le|v'|\le n<|u'|$ and $A(\xi,v')<\eps/T$, contradicting the
choice of $u'$; here $|v'|\ge q$ follows from the exclusion of
smaller denominators. Consequently,
$\lambda_1(B_\wt(\eps/T)\cap m_u^\perp,
\pi_{u'}(H_u\cap\Z^{d+1}))\ge1/4$.
Since $u'\in H_u$ and $u^*$ is primitive, $m_u$ is primitive in
$\Lambda_{u'}^*$, and $\pi_{u'}(H_u\cap\Z^{d+1})=\Lambda_{u'}\cap m_u^\perp$
has covolume $|u'|^{-1}\|m_u\|_2$. Thus
Lemma~\ref{lem:lattice}(2)--(3) gives
\begin{equation}\label{eq:sampledensity}
 \frac{\vol_{d-1}(B_\wt(\eps/T)\cap m_u^\perp)}
      {\covol\pi_{u'}(H_u\cap\Z^{d+1})}
 \ll_\wt\frac{|u'|\eps/T}{\max_i|m_{u,i}|(\eps/T)^{w_i}}
 \asymp_d\frac{|u'|}{T\Phi(T)}.
\end{equation}

Starting from $T_0$, set $T_{j+1}=\min\{L_*T_j,T_*\}$
and let $J$ be the first index with $T_J=T_*$. Define
\[
 \mathcal U_j=\bigcup_{i=0}^j\mathcal C(T_i)\quad\text{and}\quad
 S_j=\sum_{u'\in\mathcal U_j}|u'|^{-(d-1)}.
\]
Since $T_0=C_*\eps/r(u)<T_*$ and $\Phi$ decreases,
\eqref{eq:compact-density} gives
$k_uT_0/q\ll_\wt\Phi(T_0)^{-1}\ll L_*^{-1}$.
The bound \eqref{eq:hyperplane-count} and
$T_0/q\le A_*<\Phi(q)^{1/w_d}$ give, uniformly in $L_*>2$,
\begin{equation}\label{eq:sampleinitial}
 S_0\ll_\wt q^{-(d-1)}(1+\log\Phi(q)).
\end{equation}
For $v\in\mathcal C(T_{j+1})\setminus\mathcal U_j$ and any
$\xi\in\Xi$ with $A(\xi,v)\le\eps/T_{j+1}$, let
$u'\in\mathcal C(T_j)$ be the vector chosen for $\xi$ at scale $T_j$.
Since $v\notin\mathcal C(T_j)$ and $A(\xi,v)\le\eps/T_j$,
\[
 |u'|<T_j\le|v|<T_{j+1},\qquad
 |\pi_{u'}(v)_i|\le A(\xi,v)^{w_i}
   +\frac{|v|}{|u'|}A(\xi,u')^{w_i}
 \le\frac{2|v|}{|u'|}(\eps/T_j)^{w_i}.
\]
For fixed $u'$, consider all such $v$ arising from any $\xi\in\Xi$.
The projection $\pi_{u'}$ is injective when the denominator is restricted
to $[k|u'|,(k+1)|u'|)$, since
$\ker(\pi_{u'}|_{\Z^{d+1}})=\Z u'$.
The first-minimum bound, \eqref{eq:sampledensity}, and
Lemma~\ref{lem:lattice}(1) therefore bound their number by
\[
 \ll_\wt k^{d-2}+\frac{|u'|}{T_j\Phi(T_j)}k^{d-1}.
\]
Summing over $\lfloor T_j/|u'|\rfloor\le k<\lceil T_{j+1}/|u'|\rceil$
with weights $|v|^{-(d-1)}$ bounds the contribution for fixed $u'$
by
\[
 O_\wt\left(|u'|^{-(d-1)}
 \left(1+\log\frac{T_{j+1}}{T_j}
 +\frac{T_{j+1}}{T_j\Phi(T_j)}\right)\right).
\]
Every new $v$ occurs in at least one of these counts.
Since $\Phi(T_j)\ge L_*\ge T_{j+1}/T_j$ and
$\mathcal C(T_j)\subseteq\mathcal U_j$, summing over $u'$ gives
\begin{equation}\label{eq:center-mass-growth}
 S_{j+1}\le C_\wt(1+\log L_*)S_j,
\end{equation}
where $C_\wt$ is independent of $L_*$. We now fix $L_*>2$, depending
only on $d$ and $\wt$, so that
$\log(C_\wt(1+\log L_*))/\log L_*<(1-r)/2$.
There are at most $1+\log A_*/\log L_*$ sampling steps.
By \eqref{eq:sampleinitial}, \eqref{eq:sample-scale-comparison},
and \eqref{eq:center-mass-growth},
\[
 S_J\ll_\wt q^{-(d-1)}(1+\log A_*)A_*^{(1-r)/2}
       \ll_\wt q^{-(d-1)}A_*^{1-r}.
\]

For every $v\in\mathcal D_{u^*}(u)$ with $|v|\ge T_*$, let
$u'\in\mathcal C(T_*)$ be the vector chosen for $\xi_v$.
Since $|u'|<T_*\le|v|$ and $A(\xi_v,v)<A(\xi_v,u')\le\eps/T_*$,
\[
 \pi_{u'}(v)\in \bigl(2|v|/|u'|\bigr)B_\wt(\eps/T_*).
\]
Lemma~\ref{lem:lattice}(1) and \eqref{eq:sampledensity},
together with the lift count, give for $T\ge T_*$
\[
 \#\bigl\{v\in\mathcal D_{u^*}(u):
  T_*\le|v|\le T,\ u'\text{ is chosen for }\xi_v\bigr\}
 \ll_\wt (T/|u'|)^{d-1}
       +\frac{|u'|}{L_*T_*}(T/|u'|)^d.
\]
Partial summation, combined with the preceding bound on $S_J$, yields
\begin{align*}
 \sum_{\substack{v\in\mathcal D_{u^*}(u)\\ |v|\ge T_*}}
       (q/|v|)^{d+\delta}
 &\ll_\wt q^{d+\delta}S_JT_*^{-1-\delta}
               (1+(L_*\delta)^{-1})\\
 &\ll_\wt \delta^{-1}A_*^{-r-\delta}
 \ll_\wt \delta^{-1}\tau(c_u)^{-r}.
\end{align*}
Since $\Phi(q)>1=\Phi(q\tau(c_u))$, we have $\tau(c_u)>1$; as
$\tau(c_u)\ge\tau_*(u)$, it follows that
$\tau(c_u)^{-r}\le\min(1,F(u))$.
Together with the earlier contributions this is
$O_\wt(p(u))$, as required.
\end{proof}

\subsection{First-exit geometry and section counts}
If $\sigma_\eps(u)=\varnothing$, the sum in
\eqref{eq:contraction} vanishes. For the estimates through
Subsection~\ref{subsec:in-Nv}, assume $\sigma_\eps(u)\ne\varnothing$.
Then $q=|u|>1$ and
\eqref{eq:best-node-properties}--\eqref{eq:hyperplane-count} apply.
Take $w\in\sigma_\eps(u)$ and choose $\xi$ realizing this first exit.
Let $v$ be its immediate predecessor in the chosen best-approximation
sequence.
Put
\[
 x=D_{|v|/\eps}\pi_v(w)\quad\text{and}\quad
 y=D_{|w|/\eps}\pi_w(v).
\]
Then $y_i=-(|w|/|v|)^{w_i-1}x_i$.

\begin{lemma}\label{lem:normal-transport}
The map $z\mapsto D_{|v|/|w|}z$ gives an isomorphism
\[
 D_{|v|/|w|}(\Gamma_v(\eps)^*\cap x^\perp)
       =\Gamma_w(\eps)^*\cap y^\perp.
\]
It preserves primitivity in the ambient dual lattices. Moreover,
for every $z\ne0$,
\begin{equation}\label{eq:normal-time-transport}
 \tau(D_{|v|/|w|}z)=\frac{|v|}{|w|}\tau(z).
\end{equation}
\end{lemma}
\begin{proof}
Write $z=D_{\eps/|v|}m$, where $m\in\Lambda_v^*$, and let
$(m,n)\in\Z^{d+1}\cap v^\perp$ be its corresponding normal vector. Since
\[
 z\cdot x=m\cdot\pi_v(w)=(m,n)\cdot w,
\]
the condition $z\cdot x=0$ is equivalent to $(m,n)\cdot w=0$.
The dual coordinates of $(m,n)$ at $w$ are then
$D_{\eps/|w|}m=D_{|v|/|w|}z$, and the argument is reversible.
Primitivity is preserved because divisibility in either dual lattice is
equivalent to $(m,n)\in k\Z^{d+1}$ for some integer $k>1$.
Finally, $D_{t^{-1}}D_{|v|/|w|}=D_{|v|/(t|w|)}$ gives the time relation.
\end{proof}

Let $H,N\ge1$ be dyadic scales satisfying
\[
 H\le\frac{|w|}{|v|}<2H\quad\text{and}\quad
 N\le\frac{|w|}{|u|}<2N.
\]
Since \eqref{eq:DI} holds at $T=|w|$, Lemma~\ref{lem:integer-pairing}
gives $\|y\|_\infty<2$.
Lemma~\ref{lem:approximation-domain} and
$A(\xi,w)<A(\xi,v)$ give
\[
 |x_i|\le (|v|/\eps)^{w_i}
 \left(A(\xi,w)^{w_i}+\frac{|w|}{|v|}A(\xi,v)^{w_i}\right)
 \ll H\bigl(|v|r(v)/\eps\bigr)^{w_i}.
\]
Also, $|x_i|=(|w|/|v|)^{1-w_i}|y_i|\ll H^{1-w_i}$. Together these bounds
give, for some $C=C_\wt\ge1$,
\begin{equation}\label{eq:projectionbox}
 |x_i|\le CH^{1-w_i}\min(X^{w_i},1),
 \quad\text{where}\quad X=\frac{H|v|r(v)}{\eps},
\end{equation}
so that $X=1$ marks the transition scale $H=\eps/(|v|r(v))$.
Hence all exits with $H\le|w|/|v|<2H$ have their projections in
\[
 P=\prod_{i=1}^d
 [-CH^{1-w_i}\min(X^{w_i},1),CH^{1-w_i}\min(X^{w_i},1)].
\]

By \eqref{eq:normalbound} and Lemma~\ref{lem:integer-pairing},
the normal vector $w^*$ satisfies $w^*\cdot v=0$.
Lemma~\ref{lem:normal-transport} therefore gives a primitive vector
\[
 z=D_{|w|/|v|}c_w\in\Gamma_v(\eps)^*\quad\text{with}\quad z\cdot x=0.
\]

Every immediate predecessor $v$ of a first exit lies in
$\mathcal D_{u^*}(u)\cup\{u\}$. Since $v\in H_u$, we have
$m_u\in\Lambda_v^*$, so the normal vector $u^*$ has dual coordinates
\[
 c_v(u^*)\coloneqq D_{\eps/|v|}m_u=D_{|u|/|v|}c_u\in\Gamma_v(\eps)^*
\]
at $v$ (distinct in general from the vector $c_v$ chosen at $v$). Put
$M_v=\|c_v(u^*)\|_1$.
By Lemma~\ref{lem:approximation-domain} and \eqref{eq:hyperplane-count},
for every dyadic $L\ge1$ we have
\begin{equation}\label{eq:agebounds}
 \#\left\{v\in\mathcal D_{u^*}(u)\cup\{u\}:L\le\frac{|v|}{|u|}<2L\right\}
 \ll L^{d-1}+k_uL^d.
\end{equation}
Recall that $k_u\|c_u\|_1\ll\eps$; moreover,
$M_v\le\|c_u\|_1(|u|/|v|)^{w_d}$.
Once $|v|/|u|$ and $|w|/|u|$ lie in fixed dyadic intervals,
there are only $O(1)$ possible values of $H$.

Since $w\notin H_u$, the quantity
$c_v(u^*)\cdot x=u^*\cdot w$ is a nonzero integer. Put
$\gamma=(1-w_d)^{-1}>1$. By \eqref{eq:projectionbox},
$1\le|c_v(u^*)\cdot x|\ll M_vH^{1-w_d}$, and hence
\begin{equation}\label{eq:exitthreshold}
 \frac{|v|}{|u|}\ll \|c_u\|_1 N^{1-w_d}\quad\text{and}\quad
 N\gg N_0\coloneqq \|c_u\|_1^{-\gamma}.
\end{equation}

For $v\in\mathcal D_{u^*}(u)\cup\{u\}$, let
\[
 \mathcal E_u(v)=\bigl\{u_{\nu_u(\xi)}(\xi):\xi\in\Xi_\eps(u),\ u_{\nu_u(\xi)-1}(\xi)=v\bigr\}
\]
denote the set of first exits immediately preceded by $v$.
Two integer lifts of the same projection in $\Lambda_v$ differ
by an integer multiple of $v$.

\begin{lemma}\label{lem:unweighted-exits}
For every $u\in\mathcal Q$ with $\sigma_\eps(u)\ne\varnothing$ and every
$0<\delta\le1$,
\[
 \sum_{v\in\mathcal D_{u^*}(u)\cup\{u\}}
 \sum_{w\in\mathcal E_u(v)}
 \left(\frac{|u|}{|w|}\right)^{d+\delta}
 \ll_\wt \left(\frac\eps\delta+\|c_u\|_1^\gamma\right)p(u).
\]
\end{lemma}
\begin{proof}
Fix $v\in\mathcal D_{u^*}(u)\cup\{u\}$ and set
\[
 N_v(T)=\#\{w\in\mathcal E_u(v):|w|\le T|v|\}\quad\text{for }T\ge1.
\]
Replacing $H$ by $t=|w|/|v|$ in the proof of
\eqref{eq:projectionbox}, all such exits project into the increasing box
\[
 P_v(T)=\prod_{i=1}^d
 \left[-C\min\{T(|v|r(v)/\eps)^{w_i},T^{1-w_i}\},
       C\min\{T(|v|r(v)/\eps)^{w_i},T^{1-w_i}\}\right].
\]
The box $P_v(T)$ satisfies
\[
 \vol_d P_v(T)\ll_\wt T^{d-1},\qquad
 P_v(T)\subseteq CT\,B_\wt(|v|r(v)/\eps).
\]
Suppose $N_v(T)>0$, and put
$j=\dim\Span(\Gamma_v(\eps)\cap P_v(T))\ge1$.
If $j=d$, Lemma~\ref{lem:lattice}(1) and
$\covol\Gamma_v(\eps)=\eps^{-1}$ give
\[
 \#(\Gamma_v(\eps)\cap P_v(T))
 \ll_\wt\eps\vol_dP_v(T)\ll_\wt\eps T^{d-1}.
\]
If $1\le j<d$, the integer-valued form
$x\mapsto c_v(u^*)\cdot x$ is nonzero on the relevant span, since
$c_v(u^*)\cdot D_{|v|/\eps}\pi_v(w)=u^*\cdot w\ne0$ for every exit.
Each integer level has dimension $j-1$ in this span.
By $\lambda_1(B_\wt(|v|r(v)/\eps),\Gamma_v(\eps))=1$
and the translated count in Lemma~\ref{lem:lattice}(1)
(the trivial bound when $j=1$), we obtain, for every
$m\in\Z\setminus\{0\}$,
\[
 \begin{aligned}
 \#\{x\in\Gamma_v(\eps)\cap P_v(T):c_v(u^*)\cdot x=m\}
 &\ll_\wt T^{j-1},\\
 \sup_{x\in P_v(T)}|c_v(u^*)\cdot x|
 &\ll_\wt M_vT^{1-w_d}.
 \end{aligned}
\]
Thus the exit projections in this case number at most
$O_\wt(M_vT^{j-w_d})$. Since each projection has $O(T)$ integer
lifts with denominator at most $T|v|$, we obtain
\[
 \begin{aligned}
 N_v(T)&\ll_\wt
 \begin{cases}
  T\cdot(\eps T^{d-1}),&j=d,\\
  T\cdot(M_vT^{1-w_d})\cdot T^{j-1},&1\le j<d,
 \end{cases}\\
 &\ll_\wt\eps T^d+M_vT^{d-w_d}.
 \end{aligned}
\]
The estimate is trivial when $N_v(T)=0$.
Moreover, any exit $w\in\mathcal E_u(v)$ with $t=|w|/|v|$ satisfies
$1\le|u^*\cdot w|=|c_v(u^*)\cdot D_{|v|/\eps}\pi_v(w)|\ll M_vt^{1-w_d}$. Since
$M_v\le\|c_u\|_1\ll_\wt\eps^{w_d}$, we may fix $c_\wt>0$ small enough
that $T_0\coloneqq c_\wt M_v^{-\gamma}\ge1$ satisfies $N_v(T)=0$ for
$T\le T_0$.
Partial summation therefore gives
\begin{align*}
 \sum_{w\in\mathcal E_u(v)} (|v|/|w|)^{d+\delta}
 &=(d+\delta)\int_{T_0}^\infty
      N_v(T)T^{-d-\delta-1}\,dT\\
 &\ll \eps/\delta+M_vT_0^{-w_d-\delta}
 \ll\eps/\delta+M_v^\gamma.
\end{align*}
Since $M_v\le\|c_u\|_1$, summing over $v$ via
Lemma~\ref{lem:horizontal} proves the lemma.
\end{proof}

For $\eps\le\delta\le1$, \eqref{eq:stateweight} gives $p(w)\le2$ when
$F(w)<\eps$. It therefore remains to bound $\delta^{-1}\min(1,F(w))$ for
first exits $w$ with $F(w)\ge\eps$, where $c_w$ minimizes $\tau$.

For a primitive $z\in\Gamma_v(\eps)^*$ and dyadic $H\ge1$, define
\[
 \begin{aligned}
  \mathcal E_u(v;H,z)&=\bigl\{w\in\mathcal E_u(v):
   H\le |w|/|v|<2H,\ D_{|w|/|v|}c_w=z\bigr\},\\
  \mathcal C_{v,z}(H)&=H^{-d-\delta}\#\mathcal E_u(v;H,z).
 \end{aligned}
\]
Let $\ell=\operatorname{rank}(\Gamma_v(\eps)\cap z^\perp\cap P)$.
If $\mathcal C_{v,z}(H)>0$, then $1\le\ell\le d-1$ and the
integer-valued form $x\mapsto c_v(u^*)\cdot x$ is nonzero on
$\Span(\Gamma_v(\eps)\cap z^\perp\cap P)$.
As in Lemma~\ref{lem:unweighted-exits}, the inclusion
$P\subseteq CH\,B_\wt(|v|r(v)/\eps)$ and the first-minimum
normalization give at most $O_\wt(H^{\ell-1})$ projections per
nonzero integer level. The number of such levels is
$O_\wt(M_vH^{1-w_d})$, and each projection has $O(H)$ integer lifts.
Hence
\begin{equation}\label{eq:levelcount}
 \mathcal C_{v,z}(H)
 \ll_\wt H^{-d-\delta}\cdot H\cdot
      (M_vH^{1-w_d})\cdot H^{\ell-1}=M_vH^{\ell+1-d-w_d-\delta}.
\end{equation}

Since $H\le |w|/|v|<2H$, we have
$\|D_H^{-1}z\|_1\asymp\|c_w\|_1$.
When $\ell=d-1$, the section spans $z^\perp$ and
$\covol(\Gamma_v(\eps)\cap z^\perp)=\eps^{-1}\|z\|_2$.
Lemma~\ref{lem:lattice}(1)--(3), followed by the lift count, gives
\begin{equation}\label{eq:normalwidth}
 \mathcal C_{v,z}(H)
 \ll\eps H^{-1-\delta}
       \frac{\prod_i\min(X^{w_i},1)}
       {\max_i(H^{-w_i}\min(X^{w_i},1)|z_i|)}.
\end{equation}
Write $z=D_{\eps/|v|}m$ with $m\in\Lambda_v^*$, and put
\[
 B_z=\|D_{r(v)}m\|_\infty
 =\max_{1\le i\le d}\bigl(|v|r(v)/\eps\bigr)^{w_i}|z_i|.
\]
For $X\le1$, the denominator in \eqref{eq:normalwidth} equals $B_z$ and
$\eps\prod_iX^{w_i}=\eps X=H|v|r(v)$; for $X>1$, the denominator is at
least $X^{-w_1}B_z$ and $\eps=H|v|r(v)X^{-1}$. Thus
\begin{equation}\label{eq:sectioncoarea}
 \mathcal C_{v,z}(H)
 \ll |v|r(v)H^{-\delta}B_z^{-1}
             \min(1,X^{w_1-1}).
\end{equation}
Likewise, since $\min(X^{w_i},1)\ge\min(X^{w_1},1)$ for every $i$,
\eqref{eq:normalwidth} yields for all $X>0$
\begin{equation}\label{eq:normalwidth-large}
 \mathcal C_{v,z}(H)
 \ll\frac{\eps\min(1,X^{1-w_1})}{\|D_H^{-1}z\|_1}H^{-1-\delta}.
\end{equation}

\subsection{Counting normal vectors at the first exit}

For the remaining exits with $F(w)\ge\eps$, write
$A\le\tau_*(w)<2A$, with $A$ dyadic.

For fixed $v$, let $\mathcal N_v$ consist of the lines spanned by
minimizers of $\tau$ in $\Gamma_v(\eps)^*$ together with the lines
\[
 D_{\eps/|v|}\Span
 \bigl\{m\in\Lambda_v^*:\|D_{r(v)}m\|_\infty\le B\bigr\}
\]
for radii $B>0$ at which $\{m\in\Lambda_v^*:\|D_{r(v)}m\|_\infty\le B\}$
has rank one.

Every nonzero $z\in\Gamma_v(\eps)^*$ satisfies
$\|D_{\tau_*(v)^{-1}}z\|_1\ge\eps$, with equality when $z$ minimizes
$\tau$. If minimizers $z\ne z'$ satisfy
$z\equiv z'\pmod{3\Gamma_v(\eps)^*}$, then
$(z-z')/3\in\Gamma_v(\eps)^*\setminus\{0\}$ has
$\|D_{\tau_*(v)^{-1}}(z-z')/3\|_1\le\frac23\eps$, a contradiction. Thus
there are at most $3^d$ minimizers, and since the nested cubes have at
most one rank-one direction, $\mathcal N_v$ contains $O_d(1)$ lines.

We use the following integer-level bound. Let $K$ be an
origin-symmetric convex body in the ambient space of a lattice $L$,
and let $b:L\to\Z$ be a homomorphism whose real-linear extension is nonzero on $\Span(L\cap K)$.
Suppose
\[
 \lambda_1(K\cap\ker b,L\cap\ker b)\ge c>0,
\]
with the first minimum interpreted as $+\infty$ if $L\cap\ker b=\{0\}$.
Points in each integer level of $b$ differ by elements of
$(L\cap\ker b)\cap2K$; packing therefore gives $O_{d,c}(1)$ points per level.
Since $\sup_K|b|\ge1$, there are at most $2\sup_K|b|+1$ levels. Hence
\begin{equation}\label{eq:incidence}
 \#(L\cap K)\ll_{d,c}\sup_{z\in K}|b(z)|.
\end{equation}

Fix $v,H,A$ and consider exits $w\in\mathcal E_u(v)$ with
$H\le|w|/|v|<2H$, $A\le\tau_*(w)<2A$, and $\R z\notin\mathcal N_v$.
If this collection is nonempty, fix $w_*$ and put
$x_*=D_{|v|/\eps}\pi_v(w_*)$. Let
\[
 K=\left\{z:\sum_i|z_i|(4HA)^{-w_i}\le\eps\right\}.
\]
Every corresponding vector $z=D_{|w|/|v|}c_w$ lies in $K$, since
$\tau(z)=(|w|/|v|)\tau_*(w)<4HA$.
For a nonzero lattice vector $z'$ in the kernel of
$b(z)\coloneqq z\cdot x_*$, minimality at $w_*$ and
\eqref{eq:normal-time-transport} give
$\tau(z')\ge (|w_*|/|v|)\tau_*(w_*)\ge HA$. Consequently,
\[
 \sum_i|z_i'|(4HA)^{-w_i}
 \ge4^{-w_1}\sum_i|z_i'|(HA)^{-w_i}
 \ge4^{-w_1}\eps.
\]
Thus the first minimum of the kernel lattice relative to $K$ is at least $4^{-w_1}$.

Since the line spanned by $D_{|w_*|/|v|}c_{w_*}$ lies outside
$\mathcal N_v$, this vector does not minimize $\tau$ in
$\Gamma_v(\eps)^*$.
A vector with strictly smaller $\tau$ belongs to $K$ and has
nonzero pairing with $x_*$; otherwise its transport would
contradict minimality at $w_*$.
Hence \eqref{eq:incidence} and \eqref{eq:projectionbox} give
\begin{equation}\label{eq:timeincidence}
 \#(\Gamma_v(\eps)^*\cap K)
 \ll\eps H\max_i\bigl(A^{w_i}\min(X^{w_i},1)\bigr)
 \ll\eps H
 \begin{cases}A^{w_d},&A\le1,\\A^{w_1},&A\ge1.\end{cases}
\end{equation}

Since $(D_{r(v)}\Lambda_v^*)^*=D_{r(v)^{-1}}\Lambda_v$ has covolume
$(|v|r(v))^{-1}$ and $\lambda_1([-1,1]^d,D_{r(v)^{-1}}\Lambda_v)=1$,
every nonzero vector of $D_{r(v)^{-1}}\Lambda_v$ has Euclidean norm at
least one. For any primitive sublattice $U\subset D_{r(v)}\Lambda_v^*$,
\eqref{eq:orthogonalcovol} and Minkowski's first theorem therefore give
\begin{equation}\label{eq:intrinsicdual}
 \covol U=|v|r(v)\,\covol\bigl(D_{r(v)^{-1}}\Lambda_v\cap U^\perp\bigr)
 \gg |v|r(v).
\end{equation}
Applying Lemma~\ref{lem:lattice}(1) and \eqref{eq:intrinsicdual} to the
primitive sublattice $D_{r(v)}\Lambda_v^*\cap V$, where $V$ is the span
of $D_{r(v)}\Lambda_v^*\cap[-B,B]^d$, shows that if $\dim V\ge2$ and
$0<B\le1$, then
$\#(D_{r(v)}\Lambda_v^*\cap[-B,B]^d)\ll B^2/(|v|r(v))$.
For fixed $v,H$, put
\[
 \mathcal Z_v(H)=\left\{z\in\Gamma_v(\eps)^*:
 z\text{ primitive},\ \R z\notin\mathcal N_v,\
 \operatorname{rank}(\Gamma_v(\eps)\cap z^\perp\cap P)=d-1
 \right\}.
\]
Suppose $0<B_{\max}\le1$. For $k\ge0$, set
$B=2^{-k}B_{\max}$. If $z=D_{\eps/|v|}m\in\mathcal Z_v(H)$
satisfies $B/2<B_z\le B$, then
$D_{r(v)}\Lambda_v^*\cap[-B,B]^d$ contains $D_{r(v)}m$ and hence has rank
at least two by definition of $\mathcal N_v$, so there are
$O(B^2/(|v|r(v)))$ such vectors $z$. Combining this count with
\eqref{eq:sectioncoarea} and $B_z^{-1}<2B^{-1}$ gives
\[
 \sum_{\substack{z\in\mathcal Z_v(H)\\ B/2<B_z\le B}}
 \mathcal C_{v,z}(H)
 \ll BH^{-\delta}\min(1,X^{w_1-1}).
\]
Summing over $k\ge0$, using $\sum_{k\ge0}2^{-k}B_{\max}=2B_{\max}$,
yields
\begin{equation}\label{eq:intrinsiccount}
 \sum_{\substack{z\in\mathcal Z_v(H)\\ B_z\le B_{\max}}}
 \mathcal C_{v,z}(H)
 \ll B_{\max}H^{-\delta}\min(1,X^{w_1-1}).
\end{equation}

Consider $D_{r(v)}\Lambda_v^*$ in the box
$\prod_i[-bY^{w_i},bY^{w_i}]$, where $0<b\le1$, $Y>1$, and
suppose its points have rank at least two.
The full-rank count is $O_\wt(b^dY/(|v|r(v)))$.
If these points span a $j$-dimensional space $V$, $2\le j<d$,
then \eqref{eq:box-section-volume} gives
\[
 \vol_j\left(V\cap\prod_i[-bY^{w_i},bY^{w_i}]\right)
 \ll_\wt b^jY^{w_1+\cdots+w_j}
 \le b^2Y^{1-w_d}.
\]
Lemma~\ref{lem:lattice}(1) and \eqref{eq:intrinsicdual} therefore give
\[
 \#\left(D_{r(v)}\Lambda_v^*
          \cap\prod_i[-bY^{w_i},bY^{w_i}]\right)
 \ll_\wt\frac{b^dY+b^2Y^{1-w_d}}{|v|r(v)}.
\]

For $A\le\tau_*(w)<2A$ and $H\le |w|/|v|<2H$, the transported
vector $z=D_{|w|/|v|}c_w$ satisfies
$|z_i|\le\eps(4HA)^{w_i}\le4\eps(HA)^{w_i}$.
All such vectors therefore lie in the common box
\begin{equation}\label{eq:timebox}
 Q(H,A)=\prod_i[-C_0\eps(HA)^{w_i},C_0\eps(HA)^{w_i}]
\end{equation}
of $\Gamma_v(\eps)^*$, where $C_0\ge4$ is fixed.
Under $z=D_{\eps/|v|}m\mapsto D_{r(v)}m$, the box $Q(H,A)$ maps to
$\prod_i[-C_0\eps(XA)^{w_i},C_0\eps(XA)^{w_i}]$ in
$D_{r(v)}\Lambda_v^*$. When $XA>1$ and the lattice points in $Q(H,A)$
have rank at least two, applying the preceding dual-box bound with
$b=C_0\eps$ and $Y=XA$ (for $C_0\eps\le1$) and using
$(|v|r(v))^{-1}=H/(\eps X)$ gives
\begin{equation}\label{eq:timeboxcount}
 \#(\Gamma_v(\eps)^*\cap Q(H,A))
 \ll\eps^{d-1}HA+\eps HX^{-w_d}A^{1-w_d}.
\end{equation}

\begin{lemma}\label{lem:firstexit-partition}
For every $u\in\mathcal Q$ and $0<\delta\le1$,
\[
 \sum_{w\in\sigma_\eps(u)}\left(\frac{|u|}{|w|}\right)^{d+\delta}p(w)
 \le\sum_{v\in\mathcal D_{u^*}(u)\cup\{u\}}
 \sum_{w\in\mathcal E_u(v)}\left(\frac{|u|}{|w|}\right)^{d+\delta}p(w).
\]
For $v\in\mathcal D_{u^*}(u)\cup\{u\}$ and $w\in\mathcal E_u(v)$
with $F(w)\ge\eps$, put $z=D_{|w|/|v|}c_w$ and choose dyadic
scales $H\ge1$ and $A>0$ satisfying
\[
 H\le\frac{|w|}{|v|}<2H,\qquad A\le\tau_*(w)<2A.
\]
If $\R z\notin\mathcal N_v$, then both
$D_{r(v)}\Lambda_v^*\cap[-B_z,B_z]^d$ and
$\Gamma_v(\eps)^*\cap Q(H,A)$ have rank at least two.
\end{lemma}
\begin{proof}
For each $w\in\sigma_\eps(u)$, choose a sequence realizing its first
exit, and let $v$ immediately precede $w$ in that sequence.
Then $v\in\mathcal D_{u^*}(u)\cup\{u\}$ and $w\in\mathcal E_u(v)$, which
proves the stated inequality.

Write $z=D_{\eps/|v|}m$ with $m\in\Lambda_v^*$, so
$D_{r(v)}m\in[-B_z,B_z]^d$. If $D_{r(v)}\Lambda_v^*\cap[-B_z,B_z]^d$ has
rank one, its span is $\R D_{r(v)}m$, so
$\{m'\in\Lambda_v^*:\|D_{r(v)}m'\|_\infty\le B_z\}$ spans $\R m$ and its
image under $D_{\eps/|v|}$ is $\R z\in\mathcal N_v$, a contradiction.
If $F(w)\ge\eps$ and $\R z\notin\mathcal N_v$, a vector with
strictly smaller $\tau$-value is independent of the primitive $z$.
Its $\tau$-value is less than $\tau(z)<4HA$, so it lies in $Q(H,A)$.
Thus the lattice points in $Q(H,A)$ have rank at least two.
\end{proof}

\subsection{\texorpdfstring{Directions outside $\mathcal N_v$}{Directions outside Nv}}
\begin{lemma}\label{lem:ordinary-exits}
For every $u\in\mathcal Q$ with $\sigma_\eps(u)\ne\varnothing$ and every
$\eps\le\delta\le1$,
\[
 \frac1{p(u)}
 \sum_{v\in\mathcal D_{u^*}(u)\cup\{u\}}
 \sum_{\substack{w\in\mathcal E_u(v)\\F(w)\ge\eps,\,
                 \R z\notin\mathcal N_v}}
 \left(\frac{|u|}{|w|}\right)^{d+\delta}
       \delta^{-1}\min(1,F(w))
 \ll_\wt \frac\eps\delta,
\]
where $z=D_{|w|/|v|}c_w$ in each inner summand.
\end{lemma}
\begin{proof}
By Lemma~\ref{lem:horizontal}, it suffices to prove that for each fixed
$v\in\mathcal D_{u^*}(u)\cup\{u\}$,
\begin{equation}\label{eq:ordinary-fixed-v}
 \sum_{\substack{w\in\mathcal E_u(v)\\
 F(w)\ge\eps,\ \R z\notin\mathcal N_v}}
 \left(\frac{|v|}{|w|}\right)^{d+\delta}
 \min(1,F(w))\ll_\wt\eps.
\end{equation}
Partition the exits $w\in\mathcal E_u(v)$ by dyadic scales $H\ge1$ and
$A>0$ satisfying $H\le|w|/|v|<2H$ and
$A\le\tau_*(w)<2A$. Since $\min(1,F(w))\le\min(1,A^{-r})$,
we weight each group by $\min(1,A^{-r})$.

For $\ell\le d-2$, \eqref{eq:timeincidence} and
\eqref{eq:levelcount} bound the sum of $(|v|/|w|)^{d+\delta}$ for fixed
dyadic $H$ and $A$ by
\[
 CM_v\eps H^{-w_d-\delta}
 \begin{cases}A^{w_d},&A\le1,\\ A^{w_1},&A\ge1.\end{cases}
\]
The sum over $A$ converges after multiplication by
$\min(1,A^{-r})$, since $r>w_1$.
Summing over $H\gg M_v^{-\gamma}$ gives
$O_\wt(\eps M_v^\gamma)=O_\wt(\eps)$, since
$M_v\le\|c_u\|_1\ll_\wt\eps^{w_d}$.

Now suppose $\ell=d-1$. When $XA\le1$, the box
\eqref{eq:timebox} gives $B_z\le C_0\eps(XA)^{w_d}\le1$ for $C_0\eps\le1$.
By Lemma~\ref{lem:firstexit-partition} and \eqref{eq:intrinsiccount},
the contribution for fixed $H$ and $A$ is at most
\[
 C\eps(XA)^{w_d}H^{-\delta}\min(1,X^{w_1-1})\min(1,A^{-r}).
\]
Since $r>w_1\ge w_d$, summing over dyadic $A\le X^{-1}$ for fixed
$H$ gives
\[
 \sum_{A\le X^{-1}}(XA)^{w_d}\min(1,A^{-r})
 \ll_\wt\min(X^{w_d},1)
\]
for both $X\le1$ and $X>1$. Using $H^{-\delta}\le1$ and
\[
 \sum_{H\ge1}\min(X^{w_d},1)\min(1,X^{w_1-1})
 \le\sum_{H:\,X\le1}X^{w_d}+\sum_{H:\,X>1}X^{w_1-1}
 \ll_\wt1,
\]
the range $XA\le1$ contributes $O_\wt(\eps)$.

When $\ell=d-1$ and $XA>1$, since $c_w$ minimizes $\tau$ with
$A\le\tau_*(w)<2A$, the identity \eqref{eq:dualtime} gives
\[
 \eps=\sum_{i=1}^d|c_{w,i}|\tau_*(w)^{-w_i}
 \le\|c_w\|_1\max\bigl(\tau_*(w)^{-w_1},\tau_*(w)^{-w_d}\bigr),
\]
so $\|D_H^{-1}z\|_1\asymp_\wt\|c_w\|_1\gg_\wt\eps\min(A^{w_1},A^{w_d})$.
By Lemma~\ref{lem:firstexit-partition}, \eqref{eq:timeboxcount},
and \eqref{eq:normalwidth-large}, the sum of $(|v|/|w|)^{d+\delta}$
for fixed $H$ and $A$ is bounded by
\[
 H^{-\delta}
 \frac{\eps^{d-1}A+\eps\min(1,X^{1-w_1})X^{-w_d}A^{1-w_d}}
      {\min(A^{w_1},A^{w_d})}.
\]
For $b=0,w_d$, the corresponding dyadic $A$-sums are bounded by
constant multiples of
\[
 \int_0^\infty
 \frac{A^{1-b}\min(1,A^{-r})}
      {\min(A^{w_1},A^{w_d})}\,\frac{dA}{A}
 =\frac1{1-b-w_1}+\frac1{r+b+w_d-1}<\infty.
\]
The integrals converge since $1-w_1-w_d>0$ for $d\ge3$,
$r+w_d-1>0$ by \eqref{eq:upper-r}, and $1-w_1>0$.
For the first term of the bound above, summing
$H^{-\delta}$ over $H\ge1$ gives $O_\wt(\eps^{d-1}/\delta)$.
For the second term, $1-w_1-w_d>0$ again gives
\[
 \sum_{H\ge1}H^{-\delta}\min(1,X^{1-w_1})X^{-w_d}
 \le\sum_{H:\,X\le1}X^{1-w_1-w_d}+\sum_{H:\,X>1}X^{-w_d}
 \ll_\wt1.
\]
Since $d\ge3$ and $\delta\ge\eps$, the range $XA>1$ contributes
$O_\wt(\eps^{d-1}/\delta+\eps)=O_\wt(\eps)$, proving
\eqref{eq:ordinary-fixed-v}.
\end{proof}

\subsection{\texorpdfstring{Directions in $\mathcal N_v$}{Directions in Nv}}\label{subsec:in-Nv}
\begin{lemma}\label{lem:exceptional-exits}
For every $u\in\mathcal Q$ with $\sigma_\eps(u)\ne\varnothing$ and every
$\eps\le\delta\le1$,
\[
 \frac1{p(u)}
 \sum_{v\in\mathcal D_{u^*}(u)\cup\{u\}}
 \sum_{\substack{w\in\mathcal E_u(v)\\F(w)\ge\eps,\,
               \R z\in\mathcal N_v}}
 \left(\frac{|u|}{|w|}\right)^{d+\delta}
       \delta^{-1}\min(1,F(w))
 \ll_\wt \|c_u\|_1^{\gamma(1-r)},
\]
where $z=D_{|w|/|v|}c_w$ in each inner summand.
\end{lemma}
\begin{proof}
For each $v$, there are $O_d(1)$ primitive vectors
$z\in\Gamma_v(\eps)^*$ with $\R z\in\mathcal N_v$.
For every fixed primitive $z\in\Gamma_v(\eps)^*$ and dyadic $H$,
the integer-level bound \eqref{eq:levelcount} gives, uniformly in
the section rank $\ell\le d-1$,
\begin{equation}\label{eq:special-unweighted-uniform}
 \mathcal C_{v,z}(H)
 \ll M_vH^{\ell+1-d-w_d-\delta}
 \ll_\wt M_vH^{-w_d-\delta}.
\end{equation}

By \eqref{eq:normalbound}, we have $0<\|c_u\|_1\le1$ for sufficiently
small $\eps$. For dyadic scales $L,N\ge1$ satisfying
$L\le|v|/|u|<2L$ and $N\le|w|/|u|<2N$, the threshold
\eqref{eq:exitthreshold} gives
\[
 N\gg_\wt N_0=\|c_u\|_1^{-\gamma}
 \quad\text{and}\quad 1\le L\ll_\wt \|c_u\|_1 N^{1-w_d}.
\]
For fixed $N$, there are at most
$O_\wt(1+\log^+(\|c_u\|_1N^{1-w_d}))$ admissible dyadic scales $L$,
and for each $L$ only $O(1)$ possible values of $H\asymp N/L$.
Since $\|c_u\|_1N_0^{1-w_d}=1$, for every $a>0$ we obtain
\begin{equation}\label{eq:exceptional-kernel}
 \sum_{\substack{N\text{ dyadic}\\N\gg N_0}}
 N^{-a}\bigl(1+\log^+(\|c_u\|_1 N^{1-w_d})\bigr)
 \ll_{\wt,a}\int_{N_0}^\infty T^{-a}
 \bigl(1+\log(T/N_0)\bigr)\frac{dT}{T}
 \ll_{\wt,a}N_0^{-a}.
\end{equation}

If $F(u)>1$, then $\delta p(u)>1$ and $\min(1,F(w))\le1$.
By \eqref{eq:agebounds} and $k_u\ll_\wt1$, there are $O_\wt(L^d)$
vectors $v$ with $L\le|v|/|u|<2L$, each satisfying
$M_v\le \|c_u\|_1 L^{-w_d}$. Applying \eqref{eq:special-unweighted-uniform}
shows that each permitted pair $(N,L)$ contributes
$O_\wt(\|c_u\|_1 N^{-w_d})$, so \eqref{eq:exceptional-kernel}
bounds the total normalized sum by
$O_\wt(\|c_u\|_1 N_0^{-w_d})=O_\wt(\|c_u\|_1^\gamma)
=O_\wt(\|c_u\|_1^{\gamma(1-r)})$.

Assume now that $F(u)\le1$, so that $\delta p(u)=\delta+F(u)\ge\eps+F(u)$
when $\delta\ge\eps$.
Write $w^*=(m_w,n_w)$ and
$\widetilde c=D_{|w|/|u|}c_w=D_{\eps/|u|}m_w$.
If $w^*\cdot u=0$, then $\widetilde c\in\Gamma_u(\eps)^*$,
so \eqref{eq:normal-time-transport} implies
$F(w)\le(|w|/|u|)^rF(u)$.
Otherwise, choosing $\xi\in\Xi_\eps(u)$ with $u_{\nu_u(\xi)}(\xi)=w$,
we have
\[
 1\le|w^*\cdot u|
 \le\sum_i|m_{w,i}|\left((\eps/|u|)^{w_i}
 +\frac{|u|}{|w|}(\eps/|w|)^{w_i}\right)
 \le 2\|\widetilde c\|_1.
\]
For $\eps<1/2$, the identity \eqref{eq:dualtime} yields
$\tau(\widetilde c)>1$ and
$\eps\ge\frac12\tau(\widetilde c)^{-w_1}$.
In either case,
\[
 F(w)\ll_\wt (|w|/|u|)^r
       \bigl(F(u)+\eps^{r/w_1}\bigr).
\]
Since $F(u)\le1$, $\delta\ge\eps$, and $r>w_1$, we obtain, for
$w\in\mathcal E_u(v)$,
\begin{equation}\label{eq:special-ancestor-weight}
 \frac{F(w)}{\delta p(u)}\ll_\wt
 \left(\frac{|w|}{|u|}\right)^r.
\end{equation}

For $\ell=d-1$, \eqref{eq:normalwidth-large} gives
$\mathcal C_{v,z}(H)\ll\eps H^{-1-\delta}/\|D_H^{-1}z\|_1$,
while $r>w_1$, \eqref{eq:dualtime}, and $z=D_{|w|/|v|}c_w$ yield
\[
 \frac{\eps\min(1,F(w))}{\|c_w\|_1}\le F(w)
 \quad\text{and}\quad\|D_H^{-1}z\|_1\asymp_\wt\|c_w\|_1.
\]
By \eqref{eq:special-ancestor-weight}, dividing by $\delta p(u)$
gives $O_\wt((H|v|/|u|)^rH^{-1-\delta})$ for fixed $v,H,z$.
For $\ell\le d-2$, the same estimate follows from
\eqref{eq:levelcount} and $M_vH^{-w_d}\le\|c_u\|_1\le1$.
Indeed, every $w\in\mathcal E_u(v;H,z)$ satisfies
$H|v|/|u|\le |w|/|u|<2H|v|/|u|$. Therefore, for $F(u)\le1$,
independently of the section rank,
\[
 \frac1{\delta p(u)}
 \sum_{\substack{w\in\mathcal E_u(v;H,z)\\F(w)\ge\eps}}
 \left(\frac{|v|}{|w|}\right)^{d+\delta}\min(1,F(w))
 \ll_\wt \left(\frac{H|v|}{|u|}\right)^rH^{-1-\delta}.
\]

Restricting this sum to a dyadic $N$-block can only decrease it.
For $L\le |v|/|u|<2L$ and $N\le |w|/|u|<2N$, we have
$H\asymp N/L$. Set
\[
 V_L=\frac{\|c_u\|_1 L^{-w_d}H^{-w_d-\delta}}{\delta p(u)}
 \quad\text{and}\quad B=N^rH^{-1-\delta}.
\]
Combining \eqref{eq:special-unweighted-uniform} (with
$\min(1,F(w))\le1$) and the preceding bound, the normalized
contribution from each $v$ in this $L$-range and each $N$-range
is at most $C_\wt\min(V_L,B)$.
By \eqref{eq:agebounds}, the normalized contribution for fixed
$N,L$ is at most a constant multiple of
\[
 \begin{aligned}
 L^{-d-\delta}(L^{d-1}B+k_uL^dV_L)
 &=L^{-1-\delta}B+k_uL^{-\delta}V_L\\
 &\ll_\wt N^{r-1-\delta}
       +\frac\eps\delta N^{-w_d-\delta},
 \end{aligned}
\]
where $H\asymp N/L$, $k_u\|c_u\|_1\ll_\wt\eps$, and $p(u)\ge1$.
In particular, this bound is independent of $L$.
Summing this estimate over $N,L$ via \eqref{eq:exceptional-kernel} gives
$O_\wt\bigl((\eps/\delta)\|c_u\|_1^{\gamma w_d}+\|c_u\|_1^{\gamma(1-r)}\bigr)$
when $F(u)\le1$. Since $r>1-w_d$, $\delta\ge\eps$, and $\|c_u\|_1\le1$,
combining the cases $F(u)>1$ and $F(u)\le1$ proves the lemma.
\end{proof}

\subsection{Contraction and dimension}
\begin{proof}[Proof of Proposition~\ref{prop:upper}]
Let $M\ge1$, and assume $0<\eps<M^{-1}$ is sufficiently small; set
$\delta=M\eps\in[\eps,1)$.
For $u\in\mathcal Q$ with $\sigma_\eps(u)=\varnothing$, the sum over
$w\in\sigma_\eps(u)$ vanishes. Suppose $\sigma_\eps(u)\ne\varnothing$.
By Lemma~\ref{lem:firstexit-partition}, it suffices to sum over pairs
$(v,w)$ with $v\in\mathcal D_{u^*}(u)\cup\{u\}$ and $w\in\mathcal E_u(v)$.
Since $p(w)\le2$ when $F(w)<\eps$ and $p(w)=1+\delta^{-1}\min(1,F(w))$
otherwise, the contribution of $p(w)$ at exits $w$ with $F(w)<\eps$ and
the constant term $1$ at the remaining exits are bounded by twice the sum
in Lemma~\ref{lem:unweighted-exits}; after division by $p(u)$, this
contributes $O_\wt(M^{-1}+\|c_u\|_1^\gamma)$.
For exits $w$ with $F(w)\ge\eps$, the remaining term
$\delta^{-1}\min(1,F(w))$ is bounded by Lemma~\ref{lem:ordinary-exits}
outside $\mathcal N_v$ and by Lemma~\ref{lem:exceptional-exits} in
$\mathcal N_v$.

By \eqref{eq:upper-r}, \eqref{eq:normalbound}, and
$\gamma=(1-w_d)^{-1}$,
\[
 \|c_u\|_1^\gamma\le \|c_u\|_1^{\gamma(1-r)}
 \ll_\wt\eps^{w_d^2/(2(1-w_d))}
\]
uniformly in $u\in\mathcal Q$. Combining these bounds gives
\[
 \sup_{u\in\mathcal Q}\frac1{p(u)}
 \sum_{w\in\sigma_\eps(u)}
 \left(\frac{|u|}{|w|}\right)^{d+\delta}p(w)
 \le C_\wt\left(M^{-1}
   +\eps^{w_d^2/(2(1-w_d))}\right),
\]
where $C_\wt$ is independent of $M$ and $\eps$. Fix $M\ge1$ so that
$C_\wt M^{-1}<1/2$, and then take $0<\eps<M^{-1}$ sufficiently small
that the right-hand side is at most $3/4$.

The union of proper rational affine hyperplanes has Hausdorff
dimension $d-1<s_*$. For $\xi\in\DI_\wt(\eps)$ outside this union, choose
a term $u$ of $(u_n(\xi))$ with $|u|$ sufficiently large that
\eqref{eq:DI} holds for all $T\ge|u|$. Then \eqref{eq:best-node-properties}
gives $\xi\in B_\wt(u,\eps/|u|)$.
Since $\widehat v\to\xi$ along the chosen best-approximation sequence
of $\xi$, the terms of this sequence cannot all lie in $H_u$, as that
would imply $(\xi,1)\in H_u$. Thus a first exit exists and belongs to
$\sigma_\eps(u)$ by Definition~\ref{def:first-exit}.

Repeating the argument gives an infinite
$\sigma_\eps$-admissible path with the required box
membership at every term. Applying Lemma~\ref{lem:covering} with
$s=s_*+\delta/(1+w_1)\in(s_*,d)$ gives
\[
 \dim_H\DI_\wt(\eps)\le s_*+\frac{M}{1+w_1}\eps.\qedhere
\]
\end{proof}

\section{The higher-dimensional lower bound}\label{sec:lower}

We prove the lower bound in Theorem~\ref{thm:main} for $d\ge3$
by adapting the constructions of Liao--Shi--Solan--Tamam
\cite[Section~4]{LSST} and Kim--Park \cite[Sections~3--4]{KP}.

\begin{proposition}\label{prop:lower}
For $d\ge3$ and a fixed weight vector $\wt$ satisfying
\eqref{eq:weights}, there are constants $c_\wt>0$ and $\eps_0>0$ such that
\[
 \dim_H\DI_\wt(\eps)\ge s_*+c_\wt\eps
\]
for all $0<\eps<\eps_0$.
\end{proposition}

Choose $\theta>0$ so that
\begin{equation}\label{eq:lowergaps}
 \theta\le\tfrac12,\quad
 (1-w_1)(1+\theta)<1,\quad\text{and}\quad
 (1+\theta)w_2<1-w_1.
\end{equation}
The last inequality is possible because
$1-w_1-w_2=\sum_{i=3}^d w_i>0$.

\subsection{Lattice conditions and successors}
For $u\in\mathcal Q$, the horizontal projection of
$\Lambda(\widehat u,t)$ in \eqref{eq:weighted-flow} is $D_{e^t}\Lambda_u$.
Since $u$ is primitive,
\begin{equation}\label{eq:flow-farey}
 \Lambda(\widehat u,t)\cap\R e_{d+1}=|u|e^{-t}\Z e_{d+1}.
\end{equation}

Set $\eta=\eps/2$. For $r>0$, let
$K_r^*$ and $\mathcal L'$ be the following sets of unimodular
lattices in $\R^{d+1}$:
\[
 \begin{aligned}
 K_r^*&=\{\Lambda:\|\phi\|_2\ge r
                    \text{ for every }0\ne\phi\in\Lambda^*\},\\
 \mathcal L'&=\{\Lambda:\Lambda\cap\R e_{d+1}
                   =a\Z e_{d+1}\text{ for some }a\in(1/2,1]\}.
 \end{aligned}
\]
\begin{lemma}\label{lem:dual-compactness}
Let $k\ge2$, let $L\subset\R^k$ be unimodular, and suppose
$\|b\|_2\ge r>0$ for every $b\in L^*\setminus\{0\}$.
For $1\le j<k$ and independent $v_1,\ldots,v_j\in L$,
\[
 \|v_1\wedge\cdots\wedge v_j\|_2\gg_k r^{k-j}.
\]
Also, $\lambda_k(L)\ll_k r^{-1}$.
\end{lemma}
\begin{proof}
Let $U=L\cap\Span(v_1,\ldots,v_j)$ and $V=L^*\cap U^\perp$.
Every nonzero vector of the rank-$(k-j)$ lattice $V$ has
Euclidean norm at least $r$. Minkowski's first theorem gives
$\covol V\gg_k r^{k-j}$. By \eqref{eq:orthogonalcovol}
and unimodularity,
\[
 \|v_1\wedge\cdots\wedge v_j\|_2
       \ge\covol U=\covol V\gg_k r^{k-j}.
\]
To bound $\lambda_k(L)$, choose independent $v_1,\ldots,v_{k-1}\in L$
with $\|v_i\|_2\le\lambda_i(L)$.
Let $b$ be a primitive generator of
$L^*\cap\Span(v_1,\ldots,v_{k-1})^\perp$. Then
\[
 r\le\|b\|_2\le\prod_{i<k}\lambda_i(L)
                   \ll_k\lambda_k(L)^{-1},
\]
where the last bound is Minkowski's second theorem.
\end{proof}

Taking $k=d+1$ in Lemma~\ref{lem:dual-compactness}, we obtain
for $\Lambda\in K_r^*$
\begin{equation}\label{eq:lowercompact}
 \lambda_1(\Lambda)\gg r^d\quad\text{and}\quad
 \lambda_{d+1}(\Lambda)\ll r^{-1}.
\end{equation}
It also gives $\|v\wedge v'\|_2\gg r^{d-1}$ for independent
$v,v'\in\Lambda$.

By \eqref{eq:flow-farey}, $\Lambda(\widehat u,t)\in\mathcal L'$
precisely when $e^t/2<|u|\le e^t$.
For $\sigma>0$, define
\[
 b_\sigma=\diag\bigl(
 \eta^{1-w_1}e^{-(1-w_1)\sigma},
 \eta^{-w_2}e^{w_2\sigma},\ldots,
 \eta^{-w_d}e^{w_d\sigma},1\bigr).
\]
We use the same notation for its restriction to $\R^d\times\{0\}$.
\begin{definition}
For $r_0>0$, to be chosen in Proposition~\ref{prop:successors}, set
\[
 \mathcal G_\eta(\sigma)=
 \mathcal L'\cap K_{\eta^2}^*\cap b_\sigma^{-1}K_{r_0}^*.
\]
\end{definition}

For $u\in\mathcal Q$, $t\ge0$ and $\tau>0$, put
\[
 r_i(t,\tau)=\eta^{w_i}e^{-(1+w_i)t-w_i\tau}\quad\text{and}\quad
 Q(\widehat u,t,\tau)=\prod_{i=1}^d
 [\widehat u_i-r_i(t,\tau),\widehat u_i+r_i(t,\tau)].
\]
For $2\le i\le d$, these half-sides satisfy
$r_i(t,\tau)=e^{-(1+w_i)t}/(b_\tau)_{ii}$.
The first diagonal entry of $b_\tau$ makes
$\det(b_\tau|_{\R^d})=1$.

We use the closed subsets
\[
 E(\widehat u,t,\tau)=\bigl\{\xi\in Q(\widehat u,t,\tau):
 |\xi_1-\widehat u_1|\ge\tfrac12r_1(t,\tau)\bigr\}
\]
and
\[
 E^{\mathrm{in}}(\widehat u,t,\tau)=
 \left\{\xi\in\R^d:
 \begin{aligned}
  &\tfrac58r_1(t,\tau)\le|\xi_1-\widehat u_1|
                         \le\tfrac78r_1(t,\tau),\\
  &|\xi_i-\widehat u_i|\le\tfrac14r_i(t,\tau)
                         \quad\text{for }2\le i\le d
 \end{aligned}\right\}.
\]
The margins in $E^{\mathrm{in}}$ allow smaller child rectangles to lie in
$E$. The first-coordinate gap in $E$ also separates children with the
same center but different next return durations.

\begin{proposition}\label{prop:successors}
There are $r_0,c,C>0$ such that, for all sufficiently small $\eta$,
the following holds. Suppose
\[
 \tau,\sigma\ge C(1+\log(1/\eta))\quad\text{and}\quad
 (1+\theta)^{-1}\le\sigma/\tau\le1+\theta.
\]
For every $t\ge0$ and $u\in\mathcal Q$ with
$\Lambda(\widehat u,t)\in\mathcal G_\eta(\tau)$,
there are at least $c \eta e^{d\tau}$ distinct vectors $v\in\mathcal Q$
with $\widehat v\in E^{\mathrm{in}}(\widehat u,t,\tau)$ and
$\Lambda(\widehat v,t+\tau)\in\mathcal G_\eta(\sigma)$.
For every $\sigma\ge C(1+\log(1/\eta))$, there also
exist $v\in\mathcal Q$ and $t\ge0$ such that
$\Lambda(\widehat v,t)\in\mathcal G_\eta(\sigma)$.
\end{proposition}

\subsection{Counting primitive successors}
As in Liao--Shi--Solan--Tamam \cite[Section~3.2]{LSST}, we first
count primitive lattice points and then remove the points lying in
hyperplanes associated with short dual vectors.

\begin{lemma}\label{lem:primitive-box}
Let $L\subset\R^k$, $k\ge2$, be a lattice with
$1/2\le\covol L\le2$ and $\lambda_1(L)\ge r$, where $0<r\le1$.
For $a_1,\ldots,a_k>0$, put $Q=\prod_i[-a_i,a_i]$. Then
\begin{equation}\label{eq:compact-box-count}
 \#\{x\in L\cap Q:x\text{ primitive}\}
 \le2+C_k\prod_i a_i+
 C_k\sum_{j=2}^{k-1}r^{-j}
       \max_{\substack{I\subseteq\{1,\ldots,k\}\\|I|=j}}
             \prod_{i\in I}a_i.
\end{equation}
If every $a_i\ge1$, then
\begin{equation}\label{eq:compact-box-faces}
 \#\{x\in L\cap Q:x\text{ primitive}\}\le
 C_k\prod_i a_i+C_{k,r}\sum_i\prod_{j\ne i}a_j.
\end{equation}
If $a_1=T\ge\ell$ and $a_2=\cdots=a_k=\ell\le1$, the count
is at most $2+C_kr^{-k}T\ell$.
\end{lemma}
\begin{proof}
We use the standard rank-sensitive counting argument; compare
\cite[Corollary~A.3]{FK}. Put $j=\operatorname{rank}(L\cap Q)$.
If $j\le1$, there are at most two primitive vectors. If $j=k$,
Lemma~\ref{lem:lattice}(1) and $\covol L\ge1/2$ give the volume term.
For $2\le j<k$, put $V=\Span(L\cap Q)$. Since
$\lambda_1(L\cap V)\ge r$, Minkowski's first theorem gives
$\covol(L\cap V)\gg_k r^j$. Applying
Lemma~\ref{lem:lattice}(1) and \eqref{eq:box-section-volume} in $V$
gives \eqref{eq:compact-box-count}. If $a_i\ge1$ for every $i$,
each proper coordinate product is bounded by a $(k-1)$-fold product.
In the last case, each $j$-fold product with $j\ge2$ is at most
$T\ell^{j-1}\le T\ell$, and $r^{-j}\le r^{-k}$.
\end{proof}

For a unimodular lattice $\Lambda$ and $\tau>0$, put
\[
 H_i=\eta^{w_i}e^\tau\quad\text{for }1\le i\le d,
\]
and let $\mathcal C_\eta(\Lambda,\tau)$
be the set of primitive vectors $(x,y)$ of $g_\tau\Lambda$ satisfying
\[
 7/8<y\le1,\quad
 11H_1/16<|x_1|\le3H_1/4,\quad\text{and}\quad
 |x_i|\le H_i/8\quad\text{for }2\le i\le d.
\]
The shear $h_{-x/y}$ sends $(x,y)$ to $(0,y)$. By primitivity,
\[
 h_{-x/y}g_\tau\Lambda\cap\R e_{d+1}=y\Z e_{d+1},
\]
so the resulting lattice belongs to $\mathcal L'$.
If $\Lambda=\Lambda(\widehat u,t)$, there is a unique $v\in\mathcal Q$
such that $(x,y)=g_{t+\tau}h_{-\widehat u}v$, with
\[
 |v|=ye^{t+\tau}\quad\text{and}\quad
 h_{-x/y}g_\tau\Lambda=\Lambda(\widehat v,t+\tau).
\]

\begin{lemma}\label{lem:primitive-supply}
There are $c_0,C>0$ such that, for all sufficiently small $\eta$,
every $\Lambda\in K_{\eta^2}^*$ and
$\tau\ge C(1+\log(1/\eta))$ satisfy
\[
 c_0\eta e^{d\tau}\le\#\mathcal C_\eta(\Lambda,\tau)
                        \le C\eta e^{d\tau}.
\]
If $\Lambda=\Lambda(\widehat u,t)$ for $u\in\mathcal Q$ and $t\ge0$,
the candidates determine
distinct centers in $E^{\mathrm{in}}(\widehat u,t,\tau)$.
\end{lemma}
\begin{proof}
Rescale the $i$th horizontal coordinate by $H_i^{-1}$. The candidate
region becomes a fixed union of two boxes, and the transformed lattice is
\[
 L=\diag(H_1^{-1},\ldots,H_d^{-1},1)g_\tau\Lambda,
\]
whose covolume is $(\eta e^{d\tau})^{-1}$.
By \eqref{eq:lowercompact}, its successive minima satisfy
\[
 \lambda_1(L)\gg \eta^{2d} e^{-\tau}\quad\text{and}\quad
 \lambda_{d+1}(L)\ll \eta^{-2-w_1}e^{-(1-w_1)\tau}.
\]
The relative error in Lemma~\ref{lem:primitive-count} is bounded by
\[
 C_{\wt}\eta^{-2-w_1}
 (1+\tau+\log(1/\eta))e^{-(1-w_1)\tau}.
\]
For $C$ sufficiently large, this relative error is at most $1/2$.

If $\Lambda=\Lambda(\widehat u,t)$, then $\widehat v$ differs from
$\widehat u$ by $e^{-(1+w_i)(t+\tau)}x_i/y$ in coordinate $i$.
Its ratio to $r_i(t,\tau)$ is $x_i/(yH_i)$, so the candidate bounds give
$\widehat v\in E^{\mathrm{in}}(\widehat u,t,\tau)$.
Distinct primitive integer vectors with positive last coordinate
have distinct rational centers.
\end{proof}

To estimate the excluded candidates, put
\[
 \gamma_0=\frac14\min\Bigl\{
 \begin{gathered}
 w_d,\quad 1-w_1-w_2,\\
 1-(1-w_1)(1+\theta),\quad (1-w_1)-(1+\theta)w_2
 \end{gathered}\Bigr\}>0.
\]
For a sufficiently large $C=C_{\wt}$ and sufficiently small $\eta$,
\begin{equation}\label{eq:lower-error-bound}
 \begin{split}
 &\eta^{2(d-1)-w_d}(1+\log(1/\eta))\\
 &\quad+\eta^{-5}(1+\tau+\sigma+\log(1/\eta))e^{-\gamma_0\tau}
 \le\eta
 \end{split}
\end{equation}
whenever $\tau\ge C(1+\log(1/\eta))$ and
$(1+\theta)^{-1}\le\sigma/\tau\le1+\theta$.
This holds since $2(d-1)-w_d>1$,
$1+\tau+\sigma+\log(1/\eta)\ll_\theta\tau$, and $\gamma_0>0$.

For $a_1\ge\cdots\ge a_k>0$ with $k\ge2$, splitting at $a_2$ and $a_1$
gives
\begin{equation}\label{eq:capped-product-integral}
 \int_0^\infty \frac{\prod_{i=1}^k\min(a_i,t)}{t^2}\,dt
 \le\left(2+\log\frac{a_1}{a_2}\right)\prod_{i=2}^k a_i.
\end{equation}

Assume $0<\eta<1$ and
$\tau\ge C(1+\log(1/\eta))$, with $C=C_\wt$ sufficiently large.
For $\Lambda\in\mathcal L'\cap K_{\eta^2}^*$, let $a e_{d+1}$,
$1/2<a\le1$, be its primitive vertical vector and let $L$ be its
horizontal projection. Extending this vector to a lattice basis gives
\[
 \covol L=a^{-1}\quad\text{and}\quad
 L^*=\{\phi:(\phi,0)\in\Lambda^*\}.
\]
Thus $\lambda_1(L^*)\ge\eta^2$, and a vector is primitive in $L^*$
if and only if its horizontal lift is primitive in $\Lambda^*$.
By \eqref{eq:orthogonalcovol} and Minkowski's first theorem,
for $1\le j\le d$, every primitive rank-$j$ sublattice $U\subset L$ satisfies
\begin{equation}\label{eq:lower-subspace-covolume}
 \covol U=(\covol L)\covol(L^*\cap U^\perp)
             \gg_d\eta^{2(d-j)}.
\end{equation}

Set
\[
 P_i=e^{-w_i\tau}H_i=\eta^{w_i}e^{(1-w_i)\tau}\quad\text{and}\quad
 P=\prod_i[-P_i,P_i].
\]
Thus $P_1\le\cdots\le P_d$ and
$\vol P/\covol L\asymp\eta e^{(d-1)\tau}$.
For $0\ne\phi\in L^*$, the bound $\|\phi\|_2\ge\eta^2$ gives
\begin{equation}\label{eq:lower-dual-cutoff}
 \max_iP_i|\phi_i|\ge q_-\coloneqq d^{-1/2}\eta^2P_1
 =d^{-1/2}\eta^{2+w_1}e^{(1-w_1)\tau}.
\end{equation}

For any finite set $\mathcal C$ of primitive vectors of $L^*$ with
$\#\mathcal C\le C_1e^\tau$, we have
\begin{equation}\label{eq:lowerincidence}
 \begin{split}
 &\#\{x\in L\cap P\setminus\{0\}:\phi\cdot x=0
\text{ for some }\phi\in\mathcal C\}\\
 &\quad\le C\eta e^{(d-1)\tau}\left(
   \sum_{\phi\in\mathcal C}\frac1{\max_iP_i|\phi_i|}
   +\eta^{-4-w_1-w_2}e^{-(1-w_1-w_2)\tau}\right),
 \end{split}
\end{equation}
where $C$ depends only on $d$ and $C_1$.
If $L\cap P\cap\phi^\perp$ spans $\phi^\perp$,
Lemma~\ref{lem:lattice}(1)--(3) gives at most
$C_d\eta e^{(d-1)\tau}/\max_iP_i|\phi_i|$ points.
For the remaining ranks $1\le j\le d-2$, take the time threshold
large enough that $P_i\ge\eta^{-2}$ for every $i$.
Equations~\eqref{eq:box-section-volume} and
\eqref{eq:lower-subspace-covolume} then bound its point count by
\[
 C_d\eta^{-2(d-j)}\prod_{i=d-j+1}^dP_i
 \le C_d\eta^{-4}\prod_{i=3}^dP_i.
\]
The rank-zero case contributes nothing. Summing over at most
$C_1e^\tau$ covectors proves
\eqref{eq:lowerincidence}. Since $w_1+w_2<1$ and
$1-w_1-w_2\ge4\gamma_0$, its relative error is at most
$\eta^{-5}e^{-\gamma_0\tau}\le\eta$.

\begin{lemma}\label{lem:unscaled-sieve}
There is $C=C_\wt>0$ such that, for all sufficiently small $\eta$,
the following holds for every $\Lambda\in\mathcal L'\cap K_{\eta^2}^*$ and
$\tau\ge C(1+\log(1/\eta))$.
Let $\mathcal B$ consist of the candidates
$(x,y)\in\mathcal C_\eta(\Lambda,\tau)$ for which
$h_{-x/y}g_\tau\Lambda\notin K_{\eta^2}^*$. Then
\[
 \#\mathcal B\ll_\wt\eta^2 e^{d\tau}.
\]
\end{lemma}
\begin{proof}
Write $ae_{d+1}$ for the primitive vertical vector and $L$ for the
horizontal projection, as above. If a candidate $(x,y)$ is excluded,
$h_{-x/y}g_\tau\Lambda$ has a primitive dual vector of length less than
$\eta^2$.
Its pairing with the primitive vertical vector $(0,y)$ is an integer
of absolute value less than one, and hence vanishes. Pulling the
covector back through the shear gives a covector of $(g_\tau\Lambda)^*$.
For $1\le i\le d$, its coordinates satisfy
\[
 |\phi_i|\le \eta^2\quad\text{and}\quad
 |\phi_{d+1}|\le C_d\eta^2\max_iH_i.
\]
Its pairing with the parent vector $ae^{-\tau}e_{d+1}$ is an integer
of absolute value at most $C_d\eta^{2+w_d}<1$. Thus $\phi_{d+1}=0$.

Let $\mathcal C$ be the primitive vectors of $(D_{e^\tau}L)^*$
with $|\phi_i|\le\eta^2$ for all $i$. Integrating each reciprocal gives
\[
 \sum_{\phi\in\mathcal C}\frac1{\max_iH_i|\phi_i|}
 =\int_{q_-}^{\infty}\frac{\#\{\phi\in\mathcal C:
                     \max_iH_i|\phi_i|\le q\}}{q^2}\,dq.
\]
The lower limit follows from \eqref{eq:lower-dual-cutoff}, since
$D_{e^\tau}\phi\in L^*$ and
$\max_iP_i|(D_{e^\tau}\phi)_i|=\max_iH_i|\phi_i|$.
The vectors counted in the integrand pull back to $L^*$
with coordinate bounds
$e^{w_i\tau}\min(\eta^2,q/H_i)$.
Apply \eqref{eq:compact-box-count} with $r=\eta^2$.
The volume term, after $q=\eta^2 e^\tau t$, contributes at most
\[
 C_d\eta^{2(d-1)}\int_0^\infty\frac{\prod_i\min(1,t\eta^{-w_i})}{t^2}\,dt
 \ll_\wt\eta^{2(d-1)-w_d}(1+\log(1/\eta)),
\]
where the last inequality follows from \eqref{eq:capped-product-integral}
with the bounds $\eta^{w_d}\ge\cdots\ge\eta^{w_1}$.

For $I\subsetneq\{1,\ldots,d\}$ with $|I|=j\ge2$, the corresponding
term has reciprocal integral bounded by
\[
 C_d\eta^{-2}
 e^{-(1-\sum_{i\in I}w_i)\tau}
 \int_0^\infty
 \frac{\prod_{i\in I}\min(1,t\eta^{-w_i})}{t^2}\,dt.
\]
The integral is at most
$C_\wt\eta^{-\min_{i\in I}w_i}(1+\log(1/\eta))$,
and $1-\sum_{i\in I}w_i\ge w_d$.
The constant term contributes at most
$2/q_-\ll\eta^{-2-w_1}e^{-(1-w_1)\tau}$.
Since $1-w_1\ge w_d$, their sum is at most
$C_\wt\eta^{-2-w_1}(1+\log(1/\eta))e^{-w_d\tau}$.
Combining this with the volume-term bound above gives
\[
 \sum_{\phi\in\mathcal C}\frac1{\max_iH_i|\phi_i|}
 \ll_\wt
 \eta^{2(d-1)-w_d}(1+\log(1/\eta))
 +\eta^{-2-w_1}(1+\log(1/\eta))e^{-w_d\tau}.
\]

For the box with half-sides $\eta^2 e^{w_i\tau}$, the rank-$j$
terms ($2\le j<d$) in \eqref{eq:compact-box-count} are at most
$C_d e^{(1-w_d)\tau}$, whereas the volume term is
$C_d\eta^{2d}e^\tau$. Increasing $C$ so that
$e^{w_d\tau}\ge\eta^{-2d}$ gives
$\#\mathcal C\ll_d\eta^{2d}e^\tau\le e^\tau$.
For each horizontal projection, the possible last coordinates differ
by multiples of $ae^{-\tau}$. Since $a>1/2$ and the candidate interval
has length $1/8$, there are $O(e^\tau)$ such lifts.
Applying \eqref{eq:lowerincidence} to $D_{e^\tau}\mathcal C\subset L^*$,
multiplying by the lift count, and using the preceding reciprocal bound
together with \eqref{eq:lower-error-bound} gives
$\#\mathcal B\ll_\wt\eta^2e^{d\tau}$.
\end{proof}

\begin{lemma}\label{lem:rescaled-sieve}
Fix $0<r_0<1$. There is $C=C_{\wt,r_0}$ such that, for all
sufficiently small $\eta$ and all $\tau,\sigma$ with
\[
 \tau,\sigma\ge C(1+\log(1/\eta))\quad\text{and}\quad
 (1+\theta)^{-1}\le\sigma/\tau\le1+\theta,
\]
the following holds. Suppose $\Lambda\in\mathcal G_\eta(\tau)$.
Let $\mathcal B$ consist of the candidates
$(x,y)\in\mathcal C_\eta(\Lambda,\tau)$ for which
$h_{-x/y}g_\tau\Lambda\notin b_\sigma^{-1}K_{r_0}^*$.
Then
\[
 \#\mathcal B\le
 \bigl[C_dr_0^d+C_{\wt,r_0}\eta\bigr]
                       \eta e^{d\tau},
\]
where the coefficient $C_d$ is independent of $r_0$.
\end{lemma}
\begin{proof}
Let $ae_{d+1}$ be the primitive vertical vector of $\Lambda$,
and let $L$ be the horizontal projection of $\Lambda$.
Put $g=b_\sigma g_\tau$. In $g\Lambda$, the candidates lie in the box
with horizontal half-sides
\[
 R_i=
 \begin{cases}
 \eta e^{\tau-(1-w_1)\sigma},&i=1,\\
 e^{\tau+w_i\sigma},&2\le i\le d,
 \end{cases}
\]
and last half-side one. Thus $R_i=(b_\sigma)_{ii}H_i$,
$\max_iR_i=R_2$ and $\prod_iR_i=\eta e^{d\tau}$.

Since $b_\sigma$ fixes the vertical coordinate, the rescaled successor
lattice of $(x,y)\in\mathcal C_\eta(\Lambda,\tau)$ is
$h_{-b_\sigma x/y}g\Lambda$.
For an excluded candidate, choose a primitive dual vector of this
lattice with length less than $r_0<1$. Its integer pairing with $(0,y)$
has absolute value less than one, so it is orthogonal to this vector.
Pulling it back through the shear gives a primitive $\phi\in(g\Lambda)^*$
orthogonal to $(b_\sigma x,y)$. The candidate bounds give
$|(b_\sigma x)_i|/y\le R_i$, so
$|\phi_{d+1}|\le r_0\sum_iR_i\le dr_0R_2$.
In particular, for each $1\le i\le d$, we have
\begin{equation}\label{eq:target-dual-caps}
 |\phi_i|\le r_0\quad\text{and}\quad
 |\phi_{d+1}|\le(d+1) r_0 R_2.
\end{equation}

We split the covectors into the horizontal case $\phi_{d+1}=0$ and the
nonhorizontal case $\phi_{d+1}\ne0$. In the horizontal case, the
pullback to $(b_\tau L)^*$ is given by the matrix
\[
 b_\tau^{-1}D_{e^\tau}b_\sigma
 =\diag\bigl(e^{\tau-(1-w_1)\sigma},
             e^{w_2\sigma},\ldots,e^{w_d\sigma}\bigr).
\]
The parent condition $b_\tau\Lambda\in K_{r_0}^*$ gives
$\lambda_1((b_\tau L)^*)\ge r_0$, and its covolume is $a$.
Put $A=r_0 e^{\tau-(1-w_1)\sigma}$.
The condition $\max_iR_i|\phi_i|\le q$ then gives coordinate bounds
\begin{equation}\label{eq:shiftedwidths}
 \begin{cases}
  \min(A,q/\eta),&i=1,\\
  \min(r_0 e^{w_i\sigma},qe^{-\tau}),&2\le i\le d.
 \end{cases}
\end{equation}
The pullback $D_{e^\tau}b_\sigma\phi$ belongs to $L^*$ and satisfies
$\max_iP_i|(D_{e^\tau}b_\sigma\phi)_i|=\max_iR_i|\phi_i|$.
Thus \eqref{eq:lower-dual-cutoff} again excludes $q<q_-$.

The inequalities in \eqref{eq:lowergaps} ensure that
$A\ge1$, $r_0e^{w_i\sigma}\ge1$ for $i\ge2$, and $\eta A<e^\tau$,
after increasing the logarithmic threshold.
For $q\le e^\tau$, all but the first half-side in
\eqref{eq:shiftedwidths} are at most $qe^{-\tau}\le1$, whereas
the first is at least $qe^{-\tau}$. Enlarging to a box with these
common remaining half-sides, Lemma~\ref{lem:primitive-box} bounds
the primitive count by
\[
 2+C_{d,r_0}\min(A,q/\eta)qe^{-\tau}.
\]
Since $\eta A<e^\tau$,
\[
 \int_0^{e^\tau}\min(A,q/\eta)\,\frac{dq}{q}
 =A\bigl(1+\log(e^\tau/(\eta A))\bigr),
\]
and integrating the count from $q_-$ to $e^\tau$ gives
\[
 C_{\wt,r_0}\bigl(q_-^{-1}
   +e^{-(1-w_1)\sigma}(1+\sigma+\log(1/\eta))\bigr).
\]

For $q\ge e^\tau$, every width in \eqref{eq:shiftedwidths} is at least
one, and the first equals $A$. Equation~\eqref{eq:compact-box-faces} gives
\[
 C_{d,r_0}A\prod_{i=2}^d\min(r_0e^{w_i\sigma},qe^{-\tau}).
\]
Substituting $q=e^\tau t$ and using \eqref{eq:capped-product-integral}
with the two largest remaining half-sides $r_0e^{w_2\sigma}$ and
$r_0e^{w_3\sigma}$ gives
\[
 \begin{aligned}
 e^{-\tau}A\int_1^\infty
    \frac{\prod_{i=2}^d\min(r_0e^{w_i\sigma},t)}{t^2}\,dt
 &\ll_{\wt}e^{-\tau}A(1+\sigma)\prod_{i=3}^d r_0e^{w_i\sigma}\\
 &=r_0^{d-1}(1+\sigma)e^{-w_2\sigma}.
 \end{aligned}
\]
Thus the horizontal reciprocal sum is at most
\begin{equation}\label{eq:shifted-horizontal-sum}
 C_{\wt,r_0}\left[
 q_-^{-1}+(1+\sigma+\log(1/\eta))e^{-(1-w_1)\sigma}
                  +(1+\sigma)e^{-w_2\sigma}\right].
\end{equation}
The box of horizontal covectors has half-sides $A$ and
$r_0e^{w_i\sigma}$ for $2\le i\le d$, all at least one, with product
$r_0^d e^\tau$. Thus \eqref{eq:compact-box-faces} gives
$O_{d,r_0}(e^\tau)$ such primitive covectors.
Apply \eqref{eq:lowerincidence} to their pullbacks in $L^*$ and
multiply by the $O(e^\tau)$ lifts of each horizontal projection.

Let $\mathcal F$ be the set of primitive nonhorizontal covectors
$\phi\in(g\Lambda)^*$ satisfying
\eqref{eq:target-dual-caps}, and put
\[
 M(\phi)=\max\{\max_{i\le d}R_i|\phi_i|,|\phi_{d+1}|\}\quad\text{and}\quad
 N(q)=\#\{\phi\in\mathcal F:M(\phi)\le q\}.
\]
The pairing of $\phi$ with $ae^{-\tau}e_{d+1}$ is a nonzero
integer. Since $a\le1$, we have
$|\phi_{d+1}|\ge e^\tau/a\ge e^\tau$, and hence
\[
 \sum_{\phi\in\mathcal F}\frac1{M(\phi)}
 =\int_{e^\tau}^\infty\frac{N(q)}{q^2}\,dq
 =e^{-\tau}\int_1^\infty\frac{N(e^\tau t)}{t^2}\,dt.
\]

Define $c_i$ for $2\le i\le d$ and $c_{d+1}$ by
\[
 c_i=r_0e^{w_i\sigma}\quad\text{and}\quad
 c_{d+1}=(d+1)c_2.
\]
For $q=e^\tau t$ with $t\ge1$, pullback to
$(b_\tau\Lambda)^*$ gives coordinate bounds $A$ and $\min(c_i,t)$ for
$2\le i\le d+1$.
The first bound is $A$ because $\eta A<e^\tau\le q$.
The lattice is unimodular with first minimum at least $r_0$,
and all these coordinate bounds are at least one after increasing
the logarithmic threshold. Define
\[
 F(t)=\prod_{i=2}^{d+1}\min(c_i,t)\quad\text{and}\quad
 F_j(t)=\prod_{\substack{2\le i\le d+1\\i\ne j}}\min(c_i,t)
       \quad\text{for }2\le j\le d+1.
\]
Equation~\eqref{eq:compact-box-faces} gives
\[
 N(e^\tau t)\le C_dAF(t)
       +C_{d,r_0}\left(F(t)+A\sum_{j=2}^{d+1}F_j(t)\right).
\]

The two largest parameters among $c_2,\ldots,c_{d+1}$ are
$c_{d+1}$ and $c_2$, with ratio $d+1$. Put $K_d=2+\log(d+1)$. By
\eqref{eq:capped-product-integral}, the full-volume term satisfies
\[
 e^{-\tau}A\int_1^\infty\frac{F(t)}{t^2}\,dt
 \le K_de^{-\tau}A\prod_{i=2}^dc_i=K_dr_0^d,
\]
where we used
\[
 e^{-\tau}A=r_0e^{-(1-w_1)\sigma}\quad\text{and}\quad
 \prod_{i=2}^dc_i=r_0^{d-1}e^{(1-w_1)\sigma}.
\]

For the boundary term omitting the first coordinate, the same
integral gives
\[
 e^{-\tau}\int_1^\infty\frac{F(t)}{t^2}\,dt
 \le K_dr_0^{d-1}e^{-\tau+(1-w_1)\sigma}.
\]
For every $2\le j\le d+1$, the ratio of the two largest parameters
in $F_j$ is at most $(d+1)e^{(w_2-w_d)\sigma}$. The product of the remaining $d-2$ parameters is obtained by
omitting one factor from $\prod_{i=2}^dc_i$, and is therefore at most
$r_0^{d-2}e^{(1-w_1-w_d)\sigma}$. Using
\eqref{eq:capped-product-integral} and
$e^{-\tau}A=r_0e^{-(1-w_1)\sigma}$, we obtain
\[
 e^{-\tau}A\int_1^\infty\frac{F_j(t)}{t^2}\,dt
 \le e^{-\tau}A\bigl(K_d+(w_2-w_d)\sigma\bigr)
     r_0^{d-2}e^{(1-w_1-w_d)\sigma}
 \ll_d r_0^{d-1}(1+\sigma)e^{-w_d\sigma}.
\]
Consequently,
\begin{equation}\label{eq:nonhreciprocal}
 \sum_{\phi\in\mathcal F}
 \frac{1}{\max\{\max_{i\le d}R_i|\phi_i|,|\phi_{d+1}|\}}
 \le C_dr_0^d+C_{\wt,r_0}(1+\sigma)
 (e^{-w_d\sigma}+e^{-\tau+(1-w_1)\sigma}).
\end{equation}
Fix $\phi\in\mathcal F$. To apply the section estimate in
Lemma~\ref{lem:lattice}, we show that the box with half-sides
$R_1,\ldots,R_d,1$ contains $d$ independent vectors of
$g\Lambda\cap\phi^\perp$.
By \eqref{eq:lowercompact}, there are $d+1$ independent vectors
of $\Lambda$, each of length $O_d(\eta^{-2})$. Their horizontal
projections span $\R^d$, so we may choose $d$ of them,
$v_j=(x_j,y_j)$, with independent horizontal projections. Put
\[
 V_0=ae^{-\tau}e_{d+1},\quad
 m=\phi\cdot V_0\in\Z\setminus\{0\},\quad\text{and}\quad
 n_j=\phi\cdot(gv_j)\in\Z.
\]
The vectors $v'_j\coloneqq m gv_j-n_jV_0$ lie in
$g\Lambda\cap\phi^\perp$ and are independent: their horizontal
projections are $m$ times an invertible image of the $x_j$.
Writing $g_i$ for the horizontal diagonal entries of $g$, cancellation
of the vertical terms gives
\[
 (v'_j)_{d+1}=-ae^{-\tau}\sum_{i=1}^d\phi_i g_i x_{j,i}.
\]

Since $R_2=e^{\tau+w_2\sigma}$ and $V_0=ae^{-\tau}e_{d+1}$,
\eqref{eq:target-dual-caps} gives
$0<|m|\ll_d r_0e^{w_2\sigma}$. For $i\le d$,
$g_i/R_i=\eta^{-w_i}e^{-(1-w_i)\tau}$, so
\[
 \frac{|(v'_j)_i|}{R_i}
 \ll_d r_0\eta^{-2-w_i}e^{-(1-w_i)\tau+w_2\sigma}
 \ll_d r_0\eta^{-2-w_1}e^{-4\gamma_0\tau}.
\]
For the last coordinate, \eqref{eq:target-dual-caps} and the
cancellation above give
\[
 |(v'_j)_{d+1}|
 \le ae^{-\tau}r_0\sum_{i=1}^d g_i|x_{j,i}|
 \ll_d r_0\eta^{-2-w_1}e^{-4\gamma_0\tau}.
\]
Both bounds are less than one for the stipulated logarithmic time
threshold, provided $C_{\wt,r_0}$ is sufficiently large. Thus
$g\Lambda\cap\phi^\perp$ contains $d$ independent vectors in the
box with half-sides $R_1,\ldots,R_d,1$.
Lemma~\ref{lem:lattice}(1)--(3) bounds the number of lattice points
in this section by $C_d\eta e^{d\tau}/M(\phi)$. Summing over
$\phi\in\mathcal F$ and applying \eqref{eq:nonhreciprocal}
bounds the candidates excluded by nonhorizontal covectors.

Using \eqref{eq:shifted-horizontal-sum}, \eqref{eq:nonhreciprocal}
and \eqref{eq:lowerincidence} (with one $O(e^\tau)$ lift factor
only in the horizontal case), and noting
$q_-^{-1}\ll_\wt\eta^{-2-w_1}e^{-(1-w_1)\tau}$, all normalized boundary
terms are $O_{\wt,r_0}(\eta)$ by \eqref{eq:lowergaps} and
\eqref{eq:lower-error-bound}. The sole main term $C_dr_0^d$
has coefficient independent of $r_0$ by
\eqref{eq:compact-box-faces}.
\end{proof}

\begin{lemma}\label{lem:initial-state}
For every sufficiently small $r_0>0$, there is $\eta_0>0$ such
that, for every $0<\eta<\eta_0$ and $\sigma>0$, there exist
$v\in\mathcal Q$ and $t\ge0$ with
$\Lambda(\widehat v,t)\in\mathcal G_\eta(\sigma)$.
\end{lemma}
\begin{proof}
Choose a prime $q$ large enough that, for every $i$, both
$\eta^2q^{w_i}$ and $r_0(b_\sigma)_{ii}q^{w_i}$ lie in $[1,q)$.
This is possible since $0<w_i<1$.
The two boxes
\[
 Q_1=\prod_i[-\eta^2 q^{w_i},\eta^2 q^{w_i}]\quad\text{and}\quad
 Q_2=\prod_i[-r_0(b_\sigma)_{ii}q^{w_i},r_0(b_\sigma)_{ii}q^{w_i}]
\]
contain at most $3^d\eta^{2d}q$ and $3^dr_0^dq$ integer vectors,
respectively, since $\prod_i(b_\sigma)_{ii}=1$.
Every nonzero integer vector $m$ in either box
is nonzero modulo $q$, so the congruence $m\cdot p\equiv0\pmod q$
has exactly $q^{d-1}$ solutions $p\in(\Z/q\Z)^d$.
The number of $p$ satisfying at least one of these congruences is at most $3^d(\eta^{2d}+r_0^d)q^d<q^d$,
provided $r_0$ and then $\eta_0$ are sufficiently small.
Choose an integer representative $p$ satisfying none of these
congruences. The boxes contain the coordinate unit vectors, so
$p\not\equiv0\pmod q$ and $v=(p,q)$ is primitive.

Put $t=\log q$. The lattice $\Lambda(\widehat v,t)$ has
primitive vertical vector $e_{d+1}$ and hence belongs to $\mathcal L'$.
Its horizontal dual lattice is
\[
 \{(D_{q^{-1}}m,0):m\in\Z^d,\ m\cdot p\equiv0\pmod q\}.
\]
A dual vector of length less than $\eta^2<1$ must be horizontal,
since its pairing with $e_{d+1}$ is an integer. It would therefore
give a nonzero $m\in Q_1$ with $m\cdot p\equiv0\pmod q$, a contradiction.
The lattice $b_\sigma\Lambda(\widehat v,t)$ also has primitive
vertical vector $e_{d+1}$. A dual vector of length less than $r_0<1$
would similarly be horizontal, with first $d$ coordinates
$b_\sigma^{-1}D_{q^{-1}}m$, and would give such an $m\in Q_2$.
\end{proof}

\begin{proof}[Proof of Proposition~\ref{prop:successors}]
Let $c_0$ be the constant in Lemma~\ref{lem:primitive-supply},
and $C_d$ the main-term coefficient in Lemma~\ref{lem:rescaled-sieve}.
Choose $r_0$ small enough for Lemma~\ref{lem:initial-state}
and $C_dr_0^d<c_0/4$, then fix the common logarithmic time threshold
and $\eta_0>0$ so that both sieves apply and
$C_{\wt,r_0}\eta_0<c_0/4$.
By Lemmas~\ref{lem:primitive-supply}, \ref{lem:unscaled-sieve},
and~\ref{lem:rescaled-sieve}, the number of remaining candidates is at least
\[
 \bigl(c_0-C_dr_0^d-C_{\wt,r_0}\eta\bigr)\eta e^{d\tau}
 \ge\frac{c_0}{2}\eta e^{d\tau}.
\]
These candidates give distinct centers in
$E^{\mathrm{in}}(\widehat u,t,\tau)$, and the corresponding lattices
lie in $\mathcal L'\cap K_{\eta^2}^*\cap b_\sigma^{-1}K_{r_0}^*
=\mathcal G_\eta(\sigma)$.
Lemma~\ref{lem:initial-state} gives the existence assertion.
\end{proof}

\subsection{Separation and mass distribution}
\begin{proof}[Proof of Proposition~\ref{prop:lower}]
Put $\Delta=(\log4)/w_1$, and let $r_0,c>0$ be the constants from
Proposition~\ref{prop:successors}. Take $C=C_\wt\ge1$ large enough
for that proposition and the contraction, separation, and packing
estimates below.
Let $0<\eta<1$ be sufficiently small, and take $N\in\N$ and
$T\ge C(1+\log(1/\eta))$ with $N\Delta\le\theta T$; these
parameters will be fixed below. For
$1\le a\le N$, set
\[
 \tau_a=T+a\Delta,\quad\text{so that}\quad \tau_N\le(1+\theta)T.
\]
By Proposition~\ref{prop:successors}, there is a root
$\nu_0=(u_0,t_0,a_0)$ with
$\Lambda(\widehat u_0,t_0)\in\mathcal G_\eta(\tau_{a_0})$.

Given a state $\nu=(u,t,a)$ with
$\Lambda(\widehat u,t)\in\mathcal G_\eta(\tau_a)$, apply
Proposition~\ref{prop:successors} with $\tau=\tau_a$ and
$\sigma=\tau_b$ for each $b\in\{1,\ldots,N\}$.
For $C$ sufficiently large, $c\eta e^{d\tau_a}\ge2$, so we may set
\[
 m_a=\lfloor c\eta e^{d\tau_a}\rfloor
 \ge\tfrac12 c\eta e^{d\tau_a}\ge1.
\]
For each $b$, retain $m_a$ centers, and let $\mathcal S(\nu)$ consist of
the resulting $Nm_a$ states $\nu'=(v,t',b)$, where
$t'\coloneqq t+\tau_a$ and
$\Lambda(\widehat v,t')\in\mathcal G_\eta(\tau_b)$.
Let the vertices of a rooted tree be the finite $\mathcal S$-admissible
paths starting at $\nu_0$. For a vertex with terminal state
$\nu=(u,t,a)$, set
\[
 B(\nu)=E(\widehat u,t,\tau_a),\quad
 l_i(\nu)=r_i(t,\tau_a),\quad\text{and}\quad
 P_\nu(\nu')=\frac1{Nm_a}.
\]

The child-to-parent ratio of each bounding half-side is
$e^{-\tau_a-w_i\tau_b}\le e^{-(1+w_d)T}$.
For $C$ sufficiently large, these ratios are at most $1/8$.
The definitions of $E^{\mathrm{in}}$ and $E$ then give
$B(\nu')\subset B(\nu)$.

For children $(v,t',b)$ and $(v',t',b')$ in $\mathcal S(\nu)$ with
$v\ne v'$, apply the wedge bound in
$\Lambda(\widehat v,t')\in K_{\eta^2}^*$ to the images of $v$ and $v'$
under $g_{t'}h_{-\widehat v}$. Their last coordinates lie in $(1/2,1]$,
and the image of $v$ is vertical.
It follows that
\[
 \max_i e^{(1+w_i)t'}|\widehat v_i-\widehat {v'}_i|\ge c_d \eta^{2(d-1)}.
\]
Choose a coordinate $i$ realizing this separation. The ratio of
the coordinate distance to the shortest parent half-side is at least
\[
 c_d\eta^{2(d-1)-w_1}
 e^{(w_1-w_i)t-(1+w_i-w_1)\tau_a}
 \ge c_d\eta^{2(d-1)-w_1}e^{-\tau_a}.
\]
The ratio of a child half-side in this coordinate to the center
separation is at most
\[
 c_d^{-1}\eta^{w_i-2(d-1)}e^{-w_iT}\le\tfrac14
\]
for $C$ sufficiently large.

For children with the same center but indices $b<b'$, we have
$r_1(t',\tau_{b'})\le r_1(t',\tau_b)/4$.
The first-coordinate gap in $E$ separates these children by at least
$r_1(t',\tau_b)/4$. Dividing by the shortest parent half-side
bounds the relative gap below by $\frac14 e^{-\tau_a-w_1\tau_N}$.
Consequently there is $\rho\ge\exp(-C_\wt(T+\log(1/\eta)))$ such that
distinct sibling sets have sup-norm distance at least
$\rho l_1(\nu)$, independently of the parent time $t$.

For the mass estimate, fix a next index $b$ and place a packing
rectangle at each corresponding child center, with half-sides
\[
 D_i=\chi e^{-(1+w_i)t'}/(b_{\tau_b})_{ii}.
\]
For distinct children $(v,t',b)$ and $(v',t',b)$ in $\mathcal S(\nu)$,
applying the wedge bound in
$b_{\tau_b}\Lambda(\widehat v,t')\in K_{r_0}^*$ to the images of $v$ and
$v'$ under $b_{\tau_b}g_{t'}h_{-\widehat v}$ gives
\[
 \max_i (b_{\tau_b})_{ii}e^{(1+w_i)t'}
       |\widehat v_i-\widehat {v'}_i|\ge c'_d r_0^{d-1}.
\]
Fix $0<\chi\le1$, depending only on $d$ and $r_0$, with
$2\chi<c'_d r_0^{d-1}$.
The packing rectangles for each fixed $b$ have disjoint interiors.

Since $\prod_{i=1}^d(b_{\tau_b})_{ii}=1$, each packing rectangle has
volume $V'=(2\chi)^d e^{-(d+1)t'}$.
The packing rectangle at the child center has
$D_i=\chi l_i(\nu')\le l_i(\nu)$ for $i\ge2$. For $i=1$,
\eqref{eq:lowergaps} and $T\ge C(1+\log(1/\eta))$ give, for $C$
sufficiently large,
\[
 D_1/l_1(\nu)
 =\chi \eta^{-1}e^{-\tau_a+(1-w_1)\tau_b}\le1.
\]
Moreover, since $m_a\ge\frac12 c\eta e^{d\tau_a}$ and $t'=t+\tau_a$,
\[
 m_aV'\ge\tfrac12 c\chi^d\vol Q(\widehat u,t,\tau_a)
 =2^{d-1}c\chi^d\prod_{i=1}^d l_i(\nu).
\]

If a child $\nu'=(v,t',b)$ satisfies
$B(\nu')\cap B_\infty(\xi,r)\ne\varnothing$, then $\widehat v$ lies in
$Q(\widehat u,t,\tau_a)$ and satisfies
$|\widehat v_i-\xi_i|\le r+l_i(\nu')$ for every $i$, so its packing
rectangle lies in a box of volume at most
\[
 2^d\prod_{i=1}^d\bigl[\min\{l_i(\nu),r+l_i(\nu')\}+D_i\bigr].
\]
For $r\ge l_1(\nu)$, the first factor is at most $2l_1(\nu)$.
For $i\ge2$, since $D_i=\chi l_i(\nu')$ and
$l_i(\nu')\le l_i(\nu)$, we have
\[
 \bigl[\min\{l_i(\nu),r+l_i(\nu')\}+\chi l_i(\nu')\bigr]
 \min\bigl(1,r/l_i(\nu')\bigr)\le(2+\chi)\min\{l_i(\nu),r\}.
\]
For fixed $b$, the $m_a$ children have the same half-sides
$l_i(\nu')=r_i(t',\tau_b)$ and disjoint packing rectangles.
Let $K_b$ be the number of these children for which
$B(\nu')\cap B_\infty(\xi,r)\ne\varnothing$. Comparing volumes
and dividing by $m_aV'$ gives, for $r\ge l_1(\nu)$,
\[
 \frac{K_b}{m_a}\prod_{i=2}^d\min\bigl(1,r/l_i(\nu')\bigr)
 \le\frac{4(2+\chi)^{d-1}}{c\chi^d}
    \prod_{i=2}^d\min\bigl(1,r/l_i(\nu)\bigr).
\]
Since $P_\nu(\nu')=(Nm_a)^{-1}$, averaging over $b=1,\ldots,N$ yields
\begin{equation}\label{eq:loweronestep}
 \sum_{\substack{\nu'\in\mathcal S(\nu)\\
                 B(\nu')\cap B_\infty(\xi,r)\ne\varnothing}}
 P_\nu(\nu')\prod_{i=2}^d\min(1,r/l_i(\nu'))
 \le \lambda\prod_{i=2}^d\min(1,r/l_i(\nu)),
\end{equation}
where $\lambda=\max\{1,4(2+\chi)^{d-1}/(c\chi^d)\}$ depends only on
$d$, $c$, and $\chi$, hence only on $d$ and $\wt$.

Using \eqref{eq:loweronestep}, we apply Lemma~\ref{lem:rectangular-mass}
with the contraction bound $l_i(\nu')\le e^{-(1+w_d)T}l_i(\nu)$. Put
$R=\lambda^{(w_1-w_d)/(1+w_d)}$; then
\[
 \log\frac{l_d(\nu)}{l_1(\nu)}
 =(w_1-w_d)\bigl(t+\tau_a+\log(1/\eta)\bigr).
\]
Since $T\ge\log(1/\eta)$ and $\tau_N/T\le3/2$, we have
\begin{equation}\label{eq:loweraspect}
 \lambda^{\,1+\frac{\log(l_d(\nu)/l_1(\nu))}{(1+w_d)T}}
 \le\lambda^4R^{t/T}.
\end{equation}

For $d-1<s<d$, put
\[
 b_s=1+w_1(s-d)\quad\text{and}\quad p_s=s+b_s
       =d+(1+w_1)(s-s_*).
\]
For a vertex $\nu=(u,t,a)$, we have
\begin{equation}\label{eq:lower-rectangle-mass}
 l_1(\nu)^{s-d+1}\prod_{i=2}^d l_i(\nu)
          =\eta^{b_s}e^{-p_st-b_s\tau_a}.
\end{equation}

All constants $\theta,r_0,c,C,\chi,\lambda,R$ depend only on $d$ and
$\wt$, not on $N,T,$ or $\eta$. Choose $C_0>0$ such that
$(c/4)C_0>R^{1+\theta}$, and then choose $c_1>0$ satisfying
\[
 (1+w_1)c_1\frac{(1+\theta)(C_0+1)\Delta}{\theta}\le\log2,
\]
with $C_0,c_1$ depending only on $d$ and $\wt$. For all sufficiently
small $\eta>0$, Proposition~\ref{prop:successors} applies and the choices
\[
 N=\lceil C_0/\eta\rceil,\quad T=N\Delta/\theta,\quad\text{and}\quad s=s_*+c_1\eta
\]
satisfy $T\ge C(1+\log(1/\eta))$ and $d-1<s<d$. Then
$\tau_N=(1+\theta)T$ and
$\eta\tau_N\le(1+\theta)(C_0+1)\Delta/\theta$. As
$m_a\ge\frac12 c\eta e^{d\tau_a}$, we obtain, for $1\le a\le N$,
\[
 Nm_ae^{-p_s\tau_a}R^{-\tau_a/T}
 \ge\tfrac12 cN\eta e^{-(1+w_1)c_1\eta\tau_N}R^{-(1+\theta)}
 \ge\tfrac14 cC_0R^{-(1+\theta)}>1.
\]
Along a path
$\nu_0,\ldots,\nu_j$, with $\nu_k=(u_k,t_k,a_k)$,
this lower bound and
$t_{k+1}-t_k=\tau_{a_k}$ imply
\[
 \prod_{k=0}^{j-1}P_{\nu_k}(\nu_{k+1})
 \le e^{-p_s(t_j-t_0)}R^{-(t_j-t_0)/T}.
\]

By \eqref{eq:loweraspect}, the left side of
\eqref{eq:tree-rectangular-mass} is therefore at most
\[
 \lambda^4e^{p_st_0}R^{t_0/T}e^{-p_st_j}
 \le C_s\eta^{b_s}e^{-p_st_j-b_s\tau_{a_j}},
\]
where one can take
$C_s=\lambda^4e^{p_st_0}R^{t_0/T}\eta^{-b_s}e^{b_s\tau_N}$.
Thus \eqref{eq:tree-rectangular-mass} follows from
\eqref{eq:lower-rectangle-mass}, and Lemma~\ref{lem:rectangular-mass}
gives dimension at least $s$.

Let $\xi$ be the limit point of an infinite $\mathcal S$-admissible
path $\nu_j=(u_j,t_j,a_j)$. For every $X>e^{t_0}$, choose $j$
so that $e^{t_j}<X\le e^{t_{j+1}}$.
Since $\Lambda(\widehat u_j,t_j)\in\mathcal L'$ and
$\xi\in B(\nu_j)$, the vector $u_j$ satisfies
\[
 |u_j|\le e^{t_j}<X\quad\text{and}\quad
 |u_j|\,|\xi_i-(\widehat{u_j})_i|
 \le\eta^{w_i}e^{-w_it_{j+1}}
 \le(\eta/X)^{w_i}<(\eps/X)^{w_i},
\]
where we used $\eta=\eps/2$.
Thus the limit set lies in $\DI_\wt(\eps)$ and has dimension at least
$s_*+(c_1/2)\eps$.
\end{proof}

\section{The planar lower bound}\label{sec:planar}

Throughout this section, $d=2$, $w_1\ge w_2>0$, and
$w_1+w_2=1$. We adapt the planar construction of
Cheung--Chevallier \cite[Section~7]{CC} to prove the lower bound
in Theorem~\ref{thm:main}.

\begin{proposition}\label{prop:planar-lower}
Let $d=2$ and let $\wt$ satisfy \eqref{eq:weights}. There are constants
$c_\wt>0$ and $\eps_0>0$ such that
\[
 \dim_H\DI_\wt(\eps)\ge s_*+c_\wt\sqrt\eps
\]
for all $0<\eps<\eps_0$.
\end{proposition}

\subsection{Lattice estimates}
For $u\in\mathcal Q$, choose $x\in\Lambda_u$ with
$\|x\|_\wt=r(u)$ and an integer lift
$v_0=(p_0,q_0)\in\Z^2\times\Z$ of $x$, so that $\pi_u(v_0)=x$.
Since $x$ is primitive in $\Lambda_u$ and
$\Z^3/\Z u\simeq\Lambda_u$, the plane $H_u\coloneqq\Span(u,v_0)$
satisfies $H_u\cap\Z^3=\Z u+\Z v_0$.

\begin{lemma}\label{lem:planarplane}
There is a constant $C_\wt>0$ such that, for every
$u\in\mathcal Q$ with $r(u)<4^{-1/w_2}$, every
$\xi\in\R^2$ with $(\xi,1)\in H_u$, and every $t\ge|u|$, there is
$(p,q)\in\Z^2\times\N$ with $0<q\le t/2$ such that
\[
 t\|q\xi-p\|_\wt\le C_\wt |u|r(u).
\]
\end{lemma}
\begin{proof}
Since $x=p_0-q_0\widehat u$ and $(\xi,1)\in H_u$, we can write
$\xi=\widehat u+\beta x$ for some $\beta\in\R$. The shear $h_{-\xi}$
sends the integer basis $u,v_0$ of $H_u\cap\Z^3$ to
$(-|u|\beta x,|u|)$ and $((1-q_0\beta)x,q_0)$.
They generate a lattice in $\R x\times\R$ with covolume
$|u|\|x\|_2$. The rectangle
\[
 \{(\lambda x,q):|q|\le t/2,\ |\lambda|\le4|u|/t\}
\]
has area $8|u|\|x\|_2$, greater than four times the covolume.
By Minkowski's theorem, it contains the image of a nonzero
$(p,q)\in H_u\cap\Z^3$. If $q=0$, then
$p=\lambda x$ and $\|p\|_\infty\le4r(u)^{w_2}<1$,
contradicting integrality and nonvanishing. Replacing $(p,q)$
by its negative if necessary, we obtain $0<q\le t/2<t$.

Since $|u|/t\le1$ and $w_i<1$, for each coordinate we have
\[
 |q\xi_i-p_i|\le4(|u|/t)r(u)^{w_i}
 \le4\bigl(|u|r(u)/t\bigr)^{w_i},
\]
which proves the lemma with $C_\wt=4^{1/w_2}$.
\end{proof}

\subsection{Construction of nested rectangles}
Recall that $\Gamma_u(\eps)=D_{|u|/\eps}\Lambda_u$ has
covolume $\eps^{-1}$. We work with $u\in\mathcal Q$ for which
$\Gamma_u(\eps)$ contains a primitive vector
\begin{equation}\label{eq:planarstate}
 x_u\in[1/2,2]^2.
\end{equation}
For $z\in\Gamma_u(\eps)$ independent of $x_u$, the determinant
bound gives $\|z\|_2\gg\eps^{-1}$, since $\|x_u\|_2\asymp1$.
On the other hand, every nonzero vector of $\Gamma_u(\eps)$
with weighted quasi-norm at most $\|x_u\|_\wt$ has bounded Euclidean norm. Thus, for small $\eps$,
every such nonzero vector belongs to
$\Gamma_u(\eps)\cap\R x_u=\Z x_u$. For integers $k$ with $|k|>1$,
we have $\|k x_u\|_\wt>\|x_u\|_\wt$, so the only such vectors are
$\pm x_u$. Consequently, $H_u$ is the plane obtained from any lift of
$D_{|u|/\eps}^{-1}x_u$, and
\begin{equation}\label{eq:planardepth}
 |u|r(u)\asymp\eps.
\end{equation}

Fix a sufficiently small $\kappa\in(0,1)$ depending only on $\wt$. For each $u$ satisfying
\eqref{eq:planarstate}, let $B(u)$ be the closed rectangle centered
at $\widehat u$ with half-sides
\begin{equation}\label{eq:planarsides}
 l_i(u)=\kappa\eps^{w_i}|u|^{-1-w_i}.
\end{equation}
For a fixed parent $u$, use the normalized coordinates
\[
 Z(\xi)=|u|D_{|u|/\eps}(\xi-\widehat u),
 \quad\text{so that}\quad Z(B(u))=[-\kappa,\kappa]^2.
\]

For the initial state, take a sufficiently large integer $q$
and choose coprime integers $p_1,p_2$ such that, for $i=1,2$,
\[
 \eps^{w_i}q^{1-w_i}\le p_i\le2\eps^{w_i}q^{1-w_i}.
\]
Such a pair exists for large $q$ by Lemma~\ref{lem:primitive-count},
applied to $q^{-1}D_{q/\eps}\Z^2$ and $[1,2]^2$.
For $u=(p,q)$ and $i=1,2$, this gives
\[
 (D_{q/\eps}\widehat u)_i
 =\frac{p_i}{\eps^{w_i}q^{1-w_i}}\in[1,2].
\]
If $\widehat u=k(n+m\widehat u)$ with $k\ge2$, $n\in\Z^2$ and
$m\in\Z$, then $(1-km)p=kqn$. Since $\gcd(k,1-km)=1$,
$k$ divides both $p_1$ and $p_2$, a contradiction.
Thus $D_{q/\eps}\widehat u$ is primitive in $\Gamma_u(\eps)$.

Choose $0<\vartheta<1/10$, with
$6\vartheta w_2<w_1-w_2$ when $w_1>w_2$.
For $T\ge\eps^{-1/2}$, define
\begin{equation}\label{eq:planar-windows}
 \begin{split}
 \mathcal A_T&=\{e^j:j\in\Z,\ T\le j\le(1+\vartheta)T\},\\
 \mathcal H_T&=\{2^j:j\in\Z,\ e^{(1+5\vartheta)T}\le2^j
                            \le e^{(1+6\vartheta)T}\}.
 \end{split}
\end{equation}
We have $\#\mathcal A_T=\vartheta T+O(1)$ and
$\#\mathcal H_T=(\vartheta/\log2)T+O(1)$. The intervals
$[A,2A)$, for $A\in\mathcal A_T$, are pairwise disjoint, as are
$[H,2H)$, for $H\in\mathcal H_T$.

For $u$ satisfying \eqref{eq:planarstate}, $A\in\mathcal A_T$, and
$H\in\mathcal H_T$, the construction selects an intermediate vector $v$
and a successor $w$ satisfying
\[
 \begin{aligned}
 v&\in H_u, & A&\le\frac{|v|}{|u|}<2A,\\
 w&\notin H_u, & H&\le\frac{|w|}{|v|}<2H.
 \end{aligned}
\]

For each $u$ satisfying \eqref{eq:planarstate}, choose the lift
$v_0=(p_0,q_0)$ defining $H_u$ so that
$D_{|u|/\eps}\pi_u(v_0)=x_u$ and
$0\le q_0<|u|$.

For $A\in\mathcal A_T$, the intermediate vectors are
$v\coloneqq mu+kv_0$, where $m,k\in\Z$ satisfy
\begin{equation}\label{eq:planarintermediate}
 \gcd(m,k)=1,\quad
 A\le m<\frac{3A}{2},\quad\text{and}\quad
 \frac{\kappa A}{32}\le k\le\frac{\kappa A}{16}.
\end{equation}
Since $u,v_0$ form a basis of $H_u\cap\Z^3$, these vectors
belong to $\mathcal Q$, with
\[
 \alpha\coloneqq\frac{|v|}{|u|}
 =m+\frac{q_0}{|u|}k\in[A,2A)\quad\text{and}\quad
 \frac{\kappa}{64}\le\frac{k}{\alpha}\le\frac{\kappa}{16}.
\]

For each $v$ and $H\in\mathcal H_T$, take the primitive vectors
\begin{equation}\label{eq:planarprojections}
 y\in \Gamma_v(\eps)\cap
       \bigl([H^{w_2},2H^{w_2}]\times[H^{w_1},2H^{w_1}]\bigr)
\end{equation}
and all their integer lifts $w\in\Z^2\times\N$ satisfying
\[
 D_{|v|/\eps}\pi_v(w)=y\quad\text{and}\quad
 H\le h\coloneqq|w|/|v|<2H.
\]
Let $\sigma_{A,H}(u)$ be the set of all vectors $w$ obtained in
this way. Since $|w|/|u|=\alpha h$, these vectors satisfy
$AH\le|w|/|u|<4AH$. Define
\[
 \sigma(u)=\bigcup_{A\in\mathcal A_T,\ H\in\mathcal H_T}
                 \sigma_{A,H}(u).
\]

\begin{lemma}\label{lem:planar-successors}
For all sufficiently small $\eps>0$, every $T\ge\eps^{-1/2}$,
and every $u\in\mathcal Q$ satisfying \eqref{eq:planarstate},
the sets $\sigma_{A,H}(u)$ are pairwise disjoint. Every
$w\in\sigma(u)$ belongs to $\mathcal Q$, satisfies
\eqref{eq:planarstate}, and has $B(w)\subset B(u)$. Moreover,
\[
 \#\sigma_{A,H}(u)\asymp_\wt\eps A^2H^2,
\]
uniformly in $u$, $A\in\mathcal A_T$ and $H\in\mathcal H_T$.

There is a constant
$C_\wt\ge1$ such that, for every $w\in\sigma(u)$,
$\xi\in B(w)$, and $|u|\le t\le|w|$, some
$(p,q)\in\Z^2\times\N$ satisfies
\[
 0<q<t\quad\text{and}\quad \|q\xi-p\|_\wt<C_\wt\eps/t.
\]
\end{lemma}
\begin{proof}
After division by $A$, the coefficient rectangle in
\eqref{eq:planarintermediate} is
$[1,3/2)\times[\kappa/32,\kappa/16]$, of area $\kappa/64$.
Applying Lemma~\ref{lem:primitive-count} to $A^{-1}\Z^2$ shows
that the number of intermediate vectors at scale $A$ is
$\asymp_\wt A^2$, uniformly in $u$.

Since $v=mu+kv_0$ and $|v|=\alpha|u|$,
\[
 Z(\widehat v)=\frac{k}{\alpha}
 D_{|u|/\eps}\pi_u(v_0)=\frac{k}{\alpha}x_u.
\]
Thus $Z(\widehat v)\in[\kappa/128,\kappa/8]^2
\subset[-\kappa/2,\kappa/2]^2$.
For two distinct intermediate centers,
\begin{equation}\label{eq:planarspacing}
 \left|\frac{k}{\alpha}-\frac{k'}{\alpha'}\right|
 =\frac{|km'-k'm|}{\alpha\alpha'}\ge\frac1{\alpha\alpha'}.
\end{equation}
As $x_u\in[1/2,2]^2$, the corresponding intermediate centers
are separated in each $Z$-coordinate by at least
$c/(\alpha\alpha')$.

For an intermediate vector $v=mu+kv_0$, choose integers $a,b$
with $mb-ka=1$.
The vectors $v$ and $au+bv_0$ form a basis of
$H_u\cap\Z^3$, and the image of $au+bv_0$ under
$D_{|v|/\eps}\pi_v$ is the primitive vector
\[
 x_v=(\alpha^{-w_2}(x_u)_1,\alpha^{-w_1}(x_u)_2).
\]
The only primitive vectors of $\Gamma_v(\eps)$ on $\R x_v$
are $\pm x_v$, neither of which belongs to the box in
\eqref{eq:planarprojections}. Thus every lift of a primitive vector
in that box lies outside $H_u$.

The map $H^{-1}D_H$ sends the box in
\eqref{eq:planarprojections} to $[1,2]^2$. We therefore apply
Lemma~\ref{lem:primitive-count} to the lattice
$L=H^{-1}D_H\Gamma_v(\eps)$, whose covolume is $(\eps H)^{-1}$.
To estimate its successive minima, consider the primitive vector
$x=H^{-1}D_Hx_v\in L$, for which
$\|x\|_2\asymp(\alpha H)^{-w_2}$.
Complete $x$ to a basis of $L$, and subtract a multiple of $x$
from the second basis vector so that its component along $x$
has length at most $\|x\|_2/2$. Its perpendicular component has
length $\covol L/\|x\|_2$. Consequently,
\[
 \lambda_2(L)\ll \|x\|_2+\frac{\covol L}{\|x\|_2}
 \ll(\alpha H)^{-w_2}+\eps^{-1}\alpha^{w_2}H^{-w_1}
 \le C\eps^{-1}e^{-cT},
\]
where we may take $c=4\vartheta w_2$, since $\alpha\le2e^{(1+\vartheta)T}$
and $H\ge e^{(1+5\vartheta)T}$ give
$w_1(1+5\vartheta)-w_2(1+\vartheta)\ge4\vartheta w_2>0$ while the
first term is $O(e^{-2w_2T})$.

Since $T\ge\eps^{-1/2}$, the preceding bound gives $\lambda_2(L)\le1$
for sufficiently small $\eps$. Minkowski's second theorem then yields
$\lambda_1(L)\gg(\eps H)^{-1}$. Since
$H\le e^{(1+6\vartheta)T}$ and $\eps<1$, we have
$1+\log(1/\lambda_1(L))\ll_\wt1+T$.
Uniformly for $T\ge\eps^{-1/2}$,
\[
 \lambda_2(L)\bigl(1+\log(1/\lambda_1(L))\bigr)
 \ll_\wt\eps^{-1}(1+T)e^{-cT}
 \ll_\wt\eps^{-3/2}e^{-c\eps^{-1/2}}=o(1).
\]
Lemma~\ref{lem:primitive-count} therefore
gives $\asymp_\wt\eps H$ primitive vectors in
\eqref{eq:planarprojections}.

Each $y$ has a coset of integer lifts modulo $\Z v$, because
$\Z^3/\Z v\simeq\Lambda_v$. Their last coordinates differ by
multiples of $|v|$, so the interval $H\le|w|/|v|<2H$
contains exactly $H$ lifts, since $H\in\mathcal H_T$ is an integer.
Since $y$ is primitive in $\Gamma_v(\eps)$, the sublattice
$\Z v+\Z w$ is primitive in $\Z^3$. In particular, every
such $w$ belongs to $\mathcal Q$.
Thus, for each fixed $(v,H)$, the number of successors is
\[
 \#\{w\text{ arising from }(v,H)\}
 \asymp_\wt (\eps H)\cdot H=\eps H^2.
\]

Every such $w$ satisfies
\[
 x_w\coloneqq-D_{|w|/\eps}\pi_w(v)=(h^{-w_2}y_1,h^{-w_1}y_2)\in[1/2,2]^2.
\]
Since $\Z v+\Z w$ is primitive, $x_w$ is primitive in
$\Gamma_w(\eps)$, so \eqref{eq:planarstate} is preserved.

Suppose $w$ arises from two intermediate vectors $v,v'$.
The corresponding $x_w,x'_w\in[1/2,2]^2$ are primitive in
$\Gamma_w(\eps)$. If they were independent, the covolume bound would give
\[
 \eps^{-1}=\covol\Gamma_w(\eps)
 \le|\det(x_w,x'_w)|\le4,
\]
a contradiction for sufficiently small $\eps$.
Thus $x_w=x'_w$, by primitivity and positivity of their coordinates,
so $v-v'\in\Z w$. As $0<|v|,|v'|<|w|/2$, we get $v=v'$.
The disjoint intervals $[A,2A)$ and $[H,2H)$ then determine
$A$ and $H$ uniquely.
Combining the primitive projections with the $H$ lifts per projection
gives the asserted bound for $\#\sigma_{A,H}(u)$.

Since $\pi_w(v)=|v|(\widehat v-\widehat w)$, the definition of
$x_w$ gives
\begin{equation}\label{eq:planardisplacement}
 \widehat w_i-\widehat v_i
 =|v|^{-1}(\eps/|w|)^{w_i}(x_w)_i\quad\text{and}\quad
 Z_i(\widehat w)-Z_i(\widehat v)
 =\alpha^{-1-w_i}h^{-w_i}(x_w)_i.
\end{equation}
The half-sides of $Z(B(w))$ are $\kappa(\alpha h)^{-1-w_i}$.
Since $Z(\widehat v)\in[-\kappa/2,\kappa/2]^2$, the inclusion
$B(w)\subset B(u)$ follows for sufficiently large $T$ from
\[
 \alpha^{-1-w_i}h^{-w_i}(x_w)_i
       +\kappa(\alpha h)^{-1-w_i}
 \le2e^{-(1+2w_2)T}+\kappa e^{-2(1+w_2)T}<\kappa/2.
\]

Fix $\xi\in B(w)$. By \eqref{eq:planardisplacement} and
$l_i(w)=\kappa\eps^{w_i}|w|^{-1-w_i}$,
we have $|\xi_i-\widehat v_i|\ll_\wt
|v|^{-1}(\eps/|w|)^{w_i}$. Thus, for $|v|<t\le|w|$,
the vector $v$ gives the required approximation, with a suitable
$C_\wt$.

For $|u|\le t\le|v|$, we have $(\widehat v,1)\in H_u$ and
$|u|r(u)\asymp_\wt\eps$ by \eqref{eq:planardepth}.
Lemma~\ref{lem:planarplane}, applied with base vector $u$ at
$\widehat v$ and scale $t$, therefore gives
$(p,q)$ with $0<q\le t/2<t$ and $|q\widehat v_i-p_i|\ll(\eps/t)^{w_i}$.
Combining this with the preceding bound on $|\xi_i-\widehat v_i|$ gives
\[
 q|\xi_i-\widehat v_i|
 \ll(\eps/t)^{w_i}\frac q{|v|}
                         \left(\frac t{|w|}\right)^{w_i}
 \le C_\wt(\eps/t)^{w_i}.
\]
The triangle inequality gives the assertion after increasing $C_\wt$.
\end{proof}

\subsection{Separation and Hausdorff dimension}
For $w\in\sigma(u)$, put
\[
 \omega_u(w)=(|u|/|w|)^{2+1/T},\quad
 S(u)=\sum_{w\in\sigma(u)}\omega_u(w),\quad\text{and}\quad
 P_u(w)=\omega_u(w)/S(u).
\]
\begin{lemma}\label{lem:planar-distribution}
For the successors in Lemma~\ref{lem:planar-successors}, distinct
$w,w'\in\sigma(u)$ satisfy
\begin{equation}\label{eq:planar-separation}
 \dist_\infty(B(w),B(w'))\ge \rho l_1(u),
\end{equation}
with $\rho=c_\wt e^{-C_\wt T}>0$.

Moreover,
\begin{equation}\label{eq:planar-pressure}
 S(u)\ge c_\wt\eps T^2,
\end{equation}
and there is $\lambda_\wt\ge1$, independent of $u,\eps,T$,
such that, for $\xi\in\R^2$ and $r\ge l_1(u)$,
\begin{equation}\label{eq:planaronestep}
 \sum_{\substack{w\in\sigma(u)\\B(w)\cap B_\infty(\xi,r)\ne\varnothing}}
 P_u(w)\min(1,r/l_2(w))
 \le \lambda_\wt\min(1,r/l_2(u)).
\end{equation}
\end{lemma}
\begin{proof}
Write $2B(w)$ for the rectangle with twice the half-sides of
$B(w)$. For children arising from distinct intermediate vectors,
\eqref{eq:planarspacing} separates the intermediate centers in the
first $Z$-coordinate by at least $ce^{-2(1+\vartheta)T}$.
Equation~\eqref{eq:planardisplacement} bounds the displacement to
$\widehat w$ and the corresponding half-side of $B(w)$ by
\[
 Ce^{-[1+2w_1+5\vartheta w_1]T}
 \le Ce^{-(2+5\vartheta/2)T}.
\]
This is $o(e^{-2(1+\vartheta)T})$, so the doubled rectangles from
distinct intermediate vectors are disjoint for large $T$.

Suppose instead that $w\ne w'$ arise from the same intermediate
vector $v$ and $2B(w)\cap2B(w')\ne\varnothing$.
By \eqref{eq:planardisplacement}, the first half-side of $B(w)$
is at most $2\kappa/h$ times the first-coordinate displacement
from $\widehat v$. Hence these displacements are comparable for
$w,w'$ when $T$ is large. Since
$(x_w)_1,(x_{w'})_1\in[1/2,2]$, the same displacement formula gives
$|w|\asymp_\wt|w'|$.

The intersection of the doubled rectangles now gives, for every $i$,
\[
 |\pi_w(w')_i|=|w'|\,|\widehat w'_i-\widehat w_i|
 \le 2\kappa\eps^{w_i}
  \bigl(|w'|\,|w|^{-1-w_i}+|w'|^{-w_i}\bigr)
 \le C_\wt\kappa(\eps/|w|)^{w_i}.
\]
By \eqref{eq:planardepth} applied at $w$, we have
$r(w)\gg_\wt\eps/|w|$. Our choice of $\kappa$ gives
$\|\pi_w(w')\|_\wt<r(w)$. Since distinct primitive vectors
with positive last coordinate have distinct rational centers,
$\pi_w(w')\ne0$, contradicting the definition of $r(w)$.
Thus the doubled rectangles $2B(w)$ are pairwise disjoint, so
$\dist_\infty(B(w),B(w'))\ge\min\{l_1(w),l_1(w')\}$ for distinct
$w,w'\in\sigma(u)$.
Since $|w|/|u|\le4e^{(2+7\vartheta)T}$, we have
$l_1(w)/l_1(u)\ge c_\wt e^{-C_\wt T}$.

Each family $\sigma_{A,H}(u)$ has
$\asymp\eps A^2H^2$ members, each with weight
$\omega_u(w)\asymp(AH)^{-2-1/T}$. Hence
\[
 \sum_{w\in\sigma_{A,H}(u)}\omega_u(w)
 \asymp_\wt\eps(AH)^{-1/T}
 \ge c_\wt\eps,
\]
since $AH\le e^{(2+7\vartheta)T}$. Both $\mathcal A_T$ and
$\mathcal H_T$ contain $\gg_\wt T$ elements. Summing the preceding
bound over $A,H$ proves \eqref{eq:planar-pressure}.

If $w_1=w_2$, then $r\ge l_1(u)=l_2(u)$, so
\eqref{eq:planaronestep} follows from $\sum_{w\in\sigma(u)}P_u(w)=1$.
Suppose now that $w_1>w_2$. Fix $A,H$ and put
$\ell=(AH)^{-1-w_2}$. For every $w\in\sigma_{A,H}(u)$,
$l_2(w)/l_2(u)=(\alpha h)^{-1-w_2}\asymp_\wt\ell$,
so the second half-side of $Z(B(w))$ is comparable to $\ell$.
For an interval $I$ of length $R>0$, let $N_{A,H}(I)$ count
the children $w\in\sigma_{A,H}(u)$ for which
$Z(B(w))\cap(\R\times I)\ne\varnothing$.

By \eqref{eq:planardisplacement},
$|Z_2(\widehat w)-Z_2(\widehat v)|\ll_\wt H\ell$.
Thus the second coordinate of any intermediate center contributing
to $N_{A,H}(I)$ lies within $CH\ell$ of $I$.
By \eqref{eq:planarspacing}, these coordinates are separated by at least
$c_\wt A^{-2}$, so there are at most
$C_\wt(1+A^2(R+H\ell))$ such intermediate vectors.

The uniform primitive-point estimate in the proof of
Lemma~\ref{lem:planar-successors} gives $O_\wt(\eps H)$
admissible projections $y$ for each such $v$.
For each $y$, the centers corresponding to its $H$ lifts have
second $Z$-coordinates
\[
 Z_2(\widehat w)=Z_2(\widehat v)+\alpha^{-1-w_2}\frac{y_2}{h},
\]
where $h$ runs through an arithmetic progression with step one in
$[H,2H)$. Since $y_2\asymp H^{w_1}$, consecutive coordinates differ by
\[
 \alpha^{-1-w_2}y_2\left(\frac1h-\frac1{h+1}\right)
 =\frac{\alpha^{-1-w_2}y_2}{h(h+1)}\asymp_\wt\ell.
\]
Enlarging $I$ by $O_\wt(\ell)$ accounts for the child half-sides.
Hence at most $C_\wt\min\{H,1+R/\ell\}$ lifts for each $y$
can meet the strip.
Multiplying the three counts gives
\[
 N_{A,H}(I)\ll_\wt
 \eps H\bigl(1+A^2(R+H\ell)\bigr)
 \min\{H,1+R/\ell\}.
\]

By the choice of $\vartheta$, the ranges in \eqref{eq:planar-windows} give
\[
 A^2H\ell=A^{w_1}H^{-w_2}
 \ge e^{(w_1-w_2-6\vartheta w_2)T}\longrightarrow\infty.
\]
Thus $1+A^2(R+H\ell)\ll A^2(R+H\ell)$ for large $T$.
Using $(R+H\ell)\min\{H,1+R/\ell\}\le2H(R+\ell)$ and
$N_{A,H}(I)\le\#\sigma_{A,H}(u)\asymp_\wt\eps A^2H^2$, we obtain
\[
 N_{A,H}(I)\ll_\wt\eps A^2H^2\min\{1,R+\ell\}.
\]
Multiplying by $\min\{1,R/\ell\}$ gives
\[
 N_{A,H}(I)\min\{1,R/\ell\}
 \ll_\wt \eps A^2H^2\min\{1,R\}.
\]

Let $I$ be the second-coordinate projection of
$Z(B_\infty(\xi,r))$. Then $R=|I|=2\kappa r/l_2(u)$.
Within $\sigma_{A,H}(u)$, the probabilities $P_u(w)$ are comparable,
and $r/l_2(w)\asymp_\wt R/\ell$.
Dividing the preceding count by $\#\sigma_{A,H}(u)$ and multiplying
by the total probability of this family therefore gives
\[
 \sum_{\substack{w\in\sigma_{A,H}(u)\\
 B(w)\cap B_\infty(\xi,r)\ne\varnothing}}
 P_u(w)\min(1,r/l_2(w))
 \ll_\wt\min(1,r/l_2(u))
       \sum_{w\in\sigma_{A,H}(u)}P_u(w).
\]
Summing over $A,H$ proves \eqref{eq:planaronestep}.
\end{proof}

\begin{proof}[Proof of Proposition~\ref{prop:planar-lower}]
By \eqref{eq:planar-pressure}, choose $T\ge\eps^{-1/2}$ with
$T\asymp_\wt\eps^{-1/2}$ so that
$S(u)\ge c_\wt\eps T^2\ge\lambda_\wt^{2(w_1-w_2)/(1+w_2)}$ for every $u$.
Consider the rooted tree of finite $\sigma$-admissible paths
from the initial vector. A path ending at $u$ carries the rectangle
$B(u)$, its half-sides $l_i(u)$, and transition probabilities
$P_u(w)$. Lemma~\ref{lem:planar-distribution} provides the
separation \eqref{eq:planar-separation} and one-step estimate
\eqref{eq:planaronestep}. To apply Lemma~\ref{lem:rectangular-mass},
it remains to verify \eqref{eq:tree-rectangular-mass} with
\[
 s=2-\frac1{1+w_1}+\frac1{(1+w_1)T}.
\]

For sufficiently small $\eps$, the denominator ratios and half-side
contraction satisfy
\[
 e^{2T}\le |w|/|u|\le e^{4T}\quad\text{and}\quad
 l_i(w)\le e^{-2(1+w_2)T}l_i(u).
\]
For the half-sides in \eqref{eq:planarsides},
\[
 l_1(u)^{s-1}l_2(u)
 =\kappa^s\eps^{w_1(s-1)+w_2}|u|^{-2-1/T}.
\]

Along any finite $\sigma$-admissible path $u_0,\ldots,u_j$,
the transition probabilities satisfy
\[
 \prod_{k=0}^{j-1}P_{u_k}(u_{k+1})
 \le \lambda_\wt^{-2(w_1-w_2)j/(1+w_2)}
     \left(\frac{|u_0|}{|u_j|}\right)^{2+1/T}.
\]
Since $\log(l_2(u)/l_1(u))=(w_1-w_2)\log(|u|/\eps)$
and $|u_j|\le|u_0|e^{4jT}$, the remaining factor in
\eqref{eq:tree-rectangular-mass} is at most
\[
 \lambda_\wt^{\,1+\frac{\log(l_2(u_j)/l_1(u_j))}{2(1+w_2)T}}
 \le C_{\eps,\wt,u_0}
       \lambda_\wt^{2(w_1-w_2)j/(1+w_2)}.
\]
Combining these estimates verifies
\eqref{eq:tree-rectangular-mass}.

Lemma~\ref{lem:rectangular-mass} now gives a limit set of dimension
at least $s$. By Lemma~\ref{lem:planar-successors}, it is contained
in $\DI_\wt(C_\wt\eps)$. Hence
\[
 \dim_H\DI_\wt(C_\wt\eps)
 \ge2-\frac1{1+w_1}+\frac1{(1+w_1)T}
 \ge2-\frac1{1+w_1}+c_\wt\sqrt\eps.
\]
Replacing $\eps$ by $\eps/C_\wt$ proves the proposition.
\end{proof}

\begin{proof}[Proof of Theorem~\ref{thm:main}]
For $d\ge3$, combine Propositions~\ref{prop:upper} and~\ref{prop:lower}.
For $d=2$, combine \cite[Theorem~1.3]{LSST} and
Proposition~\ref{prop:planar-lower}.
\end{proof}

\subsection*{Acknowledgments}
The authors are especially grateful to Yitwah Cheung for his guidance
and many insightful discussions throughout this work, particularly for
suggesting that we study lattices and their duals simultaneously.
B.Y. thanks Taehyeong Kim, Gaurav Aggarwal, Anish Ghosh,
Chengyang Wu, Yuming Wei, and Han Zhang for helpful discussions.

\subsection*{Declaration on the use of AI}
The authors provided a proof of the Hausdorff dimension upper bound
for weighted singular vectors and ideas for further extensions.
AI tools assisted with developing, checking, and simplifying arguments
and preparing the manuscript.

\end{document}